\documentclass{amsart}

\usepackage{amsthm}
\usepackage{amsmath}
\usepackage{amssymb}
\usepackage{amsfonts}
\usepackage{latexsym}
\usepackage[all]{xy} \SelectTips{eu}{}
\usepackage{hyperref}
\usepackage{eucal}
\usepackage{xcolor}
\usepackage{tikz}
\usepackage{wasysym}
\usepackage{mathrsfs}

\newtheorem{intthm}{Theorem}[]

\newtheorem*{intque*}{Question}
\newtheorem*{intexa*}{Example}
\newcommand{\numberseries}{\bfseries}   
\newlength{\thmtopspace}                
\newlength{\thmbotspace}                
\newlength{\thmheadspace}               
\newlength{\thmindent}                  
\newtheoremstyle{bfupright head,slanted body}
{\thmtopspace}{\thmbotspace}
{\slshape}{\thmindent}{\bfseries}{.}{\thmheadspace}
{{\numberseries \thmnumber{#2\;}}\thmnote{#3}}

\newtheoremstyle{bfupright head,upright body}
{\thmtopspace}{\thmbotspace}
{\upshape}{\thmindent}{\bfseries}{.}{\thmheadspace}
{{\numberseries \thmnumber{#2\;}}\thmnote{#3}}

\newtheoremstyle{fixed bf head,slanted body}
{\thmtopspace}{\thmbotspace}{\slshape}
{\thmindent}{\bfseries}{.}{\thmheadspace}
{{\numberseries \thmnumber{#2\;}}\thmname{#1}\thmnote{ (#3)}}

\newtheoremstyle{fixed bf head,upright body}
{\thmtopspace}{\thmbotspace}{\upshape}
{\thmindent}{\bfseries}{.}{\thmheadspace}
{{\numberseries \thmnumber{#2\;}}\thmname{#1}\thmnote{ (#3)}}

\newtheoremstyle{numbered paragraph}
{\thmtopspace}{\thmbotspace}{\upshape}
{\thmindent}{\upshape}{}{\thmheadspace}
{{\numberseries \thmnumber{#2.}}}

\theoremstyle{bfupright head,slanted body}
\newtheorem{res}{}[section]             \newtheorem*{res*}{}

\theoremstyle{bfupright head,upright body}
\newtheorem{bfhpg}[res]{}               \newtheorem*{bfhpg*}{}

\theoremstyle{fixed bf head,slanted body}
\newtheorem{thm}[res]{Theorem}          \newtheorem*{thm*}{Theorem}
\newtheorem{prp}[res]{Proposition}      \newtheorem*{prp*}{Proposition}
\newtheorem{cor}[res]{Corollary}        \newtheorem*{cor*}{Corollary}
\newtheorem{lem}[res]{Lemma}            \newtheorem*{lem*}{Lemma}
         \newtheorem*{que*}{Question}
\theoremstyle{fixed bf head,upright body}
\newtheorem{dfn}[res]{Definition}       \newtheorem*{dfn*}{Definition}
\newtheorem{rmk}[res]{Remark}           \newtheorem*{rmk*}{Remark}
           \newtheorem*{exa*}{Example}
\newtheorem{setup}[res]{Setup}           \newtheorem*{setup*}{Setup}
\newtheorem{nota}[res]{Notation}           \newtheorem*{nota*}{Notation}

\theoremstyle{numbered paragraph}
\newtheorem{ipg}[res]{}

\newlength{\thmlistleft}        
\newlength{\thmlistright}       
\newlength{\thmlistpartopsep}   
\newlength{\thmlisttopsep}      
\newlength{\thmlistparsep}      
\newlength{\thmlistitemsep}     

\newcounter{eqc}
\newenvironment{eqc}{\begin{list}{\upshape (\textit{\roman{eqc}})}%
		{\usecounter{eqc}%
			\setlength{\leftmargin}{\thmlistleft}%
			\setlength{\labelwidth}{\thmlistleft}%
			\setlength{\rightmargin}{\thmlistright}%
			\setlength{\partopsep}{\thmlistpartopsep}%
			\setlength{\topsep}{\thmlisttopsep}%
			\setlength{\parsep}{\thmlistparsep}%
			\setlength{\itemsep}{\thmlistitemsep}}}%
	{\end{list}}%

\newcounter{prt}
\newenvironment{prt}{\begin{list}{\upshape (\alph{prt})}%
		{\usecounter{prt}%
			\setlength{\leftmargin}{\thmlistleft}%
			\setlength{\labelwidth}{\thmlistleft}%
			\setlength{\rightmargin}{\thmlistright}%
			\setlength{\partopsep}{\thmlistpartopsep}%
			\setlength{\topsep}{\thmlisttopsep}%
			\setlength{\parsep}{\thmlistparsep}%
			\setlength{\itemsep}{\thmlistitemsep}}}%
	{\end{list}}%

\newcounter{rqm}
	{\end{list}}%

	{\end{list}}%

\newenvironment{prf*}[1][Proof]{%
	\begin{proof}[\bf #1]
		\setcounter{equation}{0}
		}
	{\end{proof}
}

\newcommand{\pgref}[1]{\ref{#1}}

\renewcommand{\eqref}[1]{(\pgref{eq:#1})}

\newcommand{\eqclbl}[1]{{\upshape(\textit{#1})}}
\newcommand{\proofofimp}[3][:]{\mbox{\eqclbl{#2}$\!\implies\!$\eqclbl{#3}#1}}
\numberwithin{equation}{res}
\def\urltilda{\kern -.15em\lower .7ex\hbox{\~{}}\kern .04em}

\newcommand{\Mod}[1]{{_{#1}\mathsf{Mod}}}

\newcommand{\Ker}[1]{\nobreak{\operatorname{Ker}#1}}

\newcommand{\Hom}[3][]{\operatorname{Hom}_{#1}(#2,#3)}

\newcommand{\Ext}[4][]{\operatorname{Ext}_{#1}^{#2}(#3,#4)}
\newcommand{\Tor}[4][]{\operatorname{Tor}_{#2}^{#1}(#3,#4)}
\newcommand{\Coker}[1]{\nobreak{\operatorname{Coker}#1}}
\newcommand{\A}{\mathcal{A}}

\newcommand{\FF}{\mathcal{F}}

\newcommand{\X}{\mathcal{X}}
\newcommand{\Y}{\mathcal{Y}}
\newcommand{\W}{\mathcal{W}}

\newcommand{\C}{\mathcal{C}}
\newcommand{\Q}{\mathcal{Q}}
\newcommand{\Di}{\mathsf{Dif}(A)}

\newcommand{\DX}{\mathsf{dg} \mathcal{X}}
\newcommand{\DY}{\mathsf{dg} \mathcal{Y}}
\newcommand{\WY}{\widetilde{\mathcal{Y}}}
\newcommand{\WX}{\widetilde{\mathcal{X}}}

\newcommand{\id}{\mathrm{id}}
\newcommand{\kk}{\Bbbk}
\newcommand{\is}{\cong}

\newcommand{\rMod}[1]{{\sf{Mod}}_{#1}}

\def\soft#1{\leavevmode\setbox0=\hbox{h}\dimen7=\ht0\advance
	\dimen7 by-1ex\relax\if t#1\relax\rlap{\raise.6\dimen7
		\hbox{\kern.3ex\char'47}}#1\relax\else\if T#1\relax
	\rlap{\raise.5\dimen7\hbox{\kern1.3ex\char'47}}#1\relax
	\else\if d#1\relax\rlap{\raise.5\dimen7\hbox{\kern.9ex
			\char'47}}#1\relax\else\if D#1\relax\rlap{\raise.5\dimen7
		\hbox{\kern1.4ex\char'47}}#1\relax\else\if l#1\relax
	\rlap{\raise.5\dimen7\hbox{\kern.4ex\char'47}}#1\relax
	\else\if L#1\relax\rlap{\raise.5\dimen7\hbox{\kern.7ex
			\char'47}}#1\relax\else\message{accent \string\soft
		\space #1 not defined!}#1\relax\fi\fi\fi\fi\fi\fi}

\begin{document}

\title[Model structures on the category of Q-shaped modules]%
{Model structures on the category of Q-shaped modules}

\author[Y.J. Ma]{Yajun Ma}

\address{Yajun Ma: Department of Mathematics, Lanzhou Jiaotong University, Lanzhou 730070, China}
\email{yjma@mail.lzjtu.cn}

\author[P. R. Yang]{Peiru Yang$^\ast$}
\address{Peiru Yang: School of Mathematics and Statistics, Northeast Normal University, Changchun, 730070, China}
\email{yangprlife@163.com}

\thanks{$^{\ast}$ Corresponding author.}

\thanks{Y. Ma was partly supported by NSF of Gansu Province (Grant No. 23JRRA866) and the Youth Foundation of Lanzhou Jiaotong University (Grant No. 2023023).}



\makeatletter
\@namedef{subjclassname@2020}{\textup{2020} Mathematics Subject Classification}
\makeatother

\begin{abstract}
We develop a method for constructing abelian model structures on the category $\Mod{\Q,A}$ of $\Q$-shaped modules from cotorsion pairs in $\Mod{A}$, where $\Q$ is a small preadditive category satisfying certain conditions and $\Mod{A}$ denotes the category of left $A$-modules for any ring $A$.
More precisely, we construct two cotorsion pairs in  $\Mod{\Q,A}$ from a given cotorsion pair in $\Mod{A}.$
This leads to a construction of projective model structures on $\Mod{\Q,A}$ under the condition that $\Q$ has no
cycles.
We further apply this method to the category $\Di$ of differential left $A$-modules, viewed as a category of $\Q$-shaped modules for a suitable choice of $\Q$.
In this case, the induced cotorsion pairs are shown to be compatible, thereby giving rise to abelian model structures on $\Di$.
\end{abstract}
\maketitle
\section{Introduction}
\noindent
Model categories were introduced by Quillen in \cite{Qui67} to provide a uniform way of formally introducing homotopy theory into general categories. Today, model categories have become a central tool in numerous branches of mathematics, including algebraic topology, algebraic geometry, representation theory, and homological algebra;  see for instance Hovey \cite{Ho99}. A model category is a bicomplete category equipped with a model structure, which consists of three classes of morphisms, called weak equivalences, cofibrations, and fibrations respectively, satisfying certain axioms. Given a model category $\mathcal{C}$ with the class of weak equivalences $\mathcal{W}$, the associated model structure provides an explicit description of the localization $\mathcal{C}[\mathcal{W}^{-1}]$, which is equivalent to the homotopy category $\mathrm{Ho}(\mathcal C)$.

In 2002, Hovey \cite{Ho02} introduced abelian model structures, namely, model structures on abelian categories that are compatible with the underlying abelian structures, and established a bijective correspondence between abelian model structures and what are now known as Hovey triples. Based on this correspondence, Gillespie subsequently used cotorsion pairs in module categories to construct two compatible cotorsion pairs in the category of chain complexes, which in turn yield various model structures on chain complexes (see \cite{Gi04,G16}). It is well known that chain complexes over a ring $A$ can be viewed as $\Mod{A}$-valued representations of the mesh category $\Q$ of the repetitive quiver of $\overrightarrow{\mathbb{A}_{2}}=\bullet\rightarrow \bullet$. More precisely, the category $\mathrm{Ch}(A)$ of chain complexes of left $A$-modules is equivalent to the category $\Mod{\Q,A}$ of additive functors from $\Q$ to $\Mod{A}$, where $\Mod{A}$ is the category of left $A$-modules.

Recently, inspired by an observation of Iyama and Minamoto \cite{Iyama}, Holm and J{\o}rgensen constructed projective and injective model structures on the functor category $\Mod{\Q,A}$, respectively. These two model structures have the same homotopy category, which is called $\Q$-shaped derived category when $\Q$ is another suitable category equipped with a Serre functor (see \cite{HJ21,HJ23,HJ25}).

Motivated by these developments, in the present paper we develop a method for constructing model structures on $\Mod{\Q,A}$ via cotorsion pairs in $\Mod{A}.$
Specifically, we investigate How cotorsion pairs in $\Mod{A}$
can be used to construct cotorsion pairs in $\Mod{\Q,A}$. Unlike the case of the category of chain complexes, however, we cannot prove in general that the two induced cotorsion pairs are compatible. Nevertheless, we prove that, under suitable assumptions, a model structure on $\Mod{\Q,A}$ can still be constructed from a projective cotorsion pair in $\Mod{A}$ (see Theorem \ref{thm:projective model structure}).

To state our main results precisely, we first introduce the necessary notation and definitions.
Let $(\X,\Y)$ be a cotorsion pair in $\Mod{A}$ for any ring $A$. We define the following two classes in $\Mod{\Q,A}$ (see Definition \ref{df:cotorsion pair} below).
$$\quad \ \ \widetilde{\mathcal{Y}}= \{Y\in\Mod{\Q,A}\ |\ \Ext[\Q,A]{1}{S_{q}(C)}{Y}=0 ~\text{and} \ Y(q)\in \mathcal{Y} {\rm \ for~ any \ }C\in\mathcal{X}, ~q\in \Q \  \},$$
$$\quad \ \ \widetilde{\mathcal{X}}=\{X\in\Mod{\Q,A}\ |\ \Ext[\Q,A]{1}{X}{S_{q}(C)}=0\ \text{and}\ X(q)\in \mathcal{X} \text{ for any } C\in {\mathcal{Y}},~ q\in \Q \  \}.$$
We then set $\mathsf{dg} \mathcal{X}={^{\bot}\widetilde{\mathcal{Y}}}$ and $\mathsf{dg} \mathcal{Y}=\widetilde{\mathcal{X}}^{\bot}$.
For the notation of cycles in the category $\Q$, we refer to \cite[Definition 5.13]{HJ23} under the assumptions in Setup \ref{setup}.
The specific result is as follows, which is the special case of Proposition \ref{prop:cotorsion pair}, Theorems \ref{thm:complete cotorsion pair} and \ref{thm:projective model structure}.

\begin{intthm}\label{thmA}
Let $(\mathcal{X},\mathcal{Y})$ be a cotorsion pair in $\Mod{A}$.
Then $(\mathsf{dg} \mathcal{X},\widetilde{\mathcal{Y}})$ and $(\widetilde{\mathcal{X}},\mathsf{dg} \mathcal{Y})$ are cotorsion pairs in $\Mod{\Q,A}$.
 Moreover, assume that $(\mathcal{X},\mathcal{Y})$ is a projective cotorsion pair (Definition \ref{df:projective cotorsion pair}) generated by a set $\sf S$ containing a generator $U$ in $\Mod{A}$, and that
 $\Q$ has no cycles.
  Then $(\sf dg \mathcal{X},\widetilde{\mathcal{Y}})$ is a projective cotorsion pair in $\Mod{\Q,A}.$
 Consequently, there is an abelian model structure on $\Mod{\Q,A}$ in which $\mathsf{dg}\mathcal{X}$ is the class of cofibrant objects, $\widetilde{\mathcal{Y}}$ is the class of trivial objects, and every object is fibrant.
The corresponding homotopy category is equivalent to the stable category $$\mathsf{dg}\mathcal{X}/_{\Q,A}\mathsf{Prj},$$ where ${_{\Q,A}\mathsf{Prj}}$ denotes the full subcategory of projective objects in $\Mod{\Q,A}$.
\end{intthm}

We remark that, in the special case of chain complexes, the category $\Q$ has no cycles, and hence the assumption in Theorem \ref{thmA} is automatically satisfied.
However, this condition fails for the category of differential modules.
Indeed, when $\Q$ is taken to be the path category of the Jordan quiver with the relation $\varepsilon^{2}=0$,
\vspace{-2mm}
\begin{center}
\begin{tikzpicture}[baseline=-0.5ex]
\node (v) at (0,0) {$\bullet$};
\draw[->]
(v) to[out=35,in=-35,looseness=7]
node[right] {$\varepsilon$} (v);
\end{tikzpicture}
\end{center}
\vspace{-2mm}
 the category $\Mod{\Q,A}$ coincides with the category $\Di$ of differential left $A$-modules (see \cite{HJ23,HJ25}). For further results on differential modules, we refer the reader to \cite{LLAHBF2007,DN,RZ,Wei}. In this situation, we successfully induce two compatible cotorsion pairs in $\Di$ from cotorsion pairs in $\Mod{A}$, and consequently obtain an associated abelian model structure. The precise statement is the following, which is contained in Theorem \ref{thm:main result in D}.

\begin{intthm}
Let $(\mathcal{X},\mathcal{Y})$ be a complete hereditary cotorsion pair in $\Mod{A}$.
Then $(\widetilde{\mathcal{X}},\mathsf{dg}\mathcal{Y})$ and $(\mathsf{dg}\mathcal{X},\widetilde{\mathcal{Y}})$ are complete hereditary, and compatible cotorsion pairs in $\Di$.
Consequently, there is an abelian model structure on $\Di$ such that $\mathsf{dg}\mathcal{X}$, $\mathsf{dg}\mathcal{Y}$ and $\mathscr{E}$ are respectively  the classes of cofibrant objects, fibrant objects, and
 trivial objects, where $\mathscr{E}$ denotes the full subcategory of exact differential left $A$-modules in $\Di.$
The corresponding homotopy category is the $\Q$-shaped derived category, where $\Q$ is the path category of the Jordan quiver
\vspace{-2mm}
\[
\begin{tikzpicture}[baseline=-0.5ex]
\node (v) at (0,0) {$\bullet$};
\draw[->]
(v) to[out=35,in=-35,looseness=7]
node[right] {$\varepsilon$} (v);
\end{tikzpicture}
\]
\vspace{-2mm}
subject to the relation $\varepsilon^{2}=0$.
\end{intthm}

The paper is organized as follows. In Section 2, we recall some notations, the definitions of cotorsion pairs and their relationship with abelian model structures, and some basic homological properties used throughout the paper. In Section 3, we construct cotorsion pairs in $\Mod{\Q,A}$ from cotorsion pairs in $\Mod{A}.$  In particular,  we establish projective model structures on $\Mod{\Q,A}$ arising from projective cotorsion pairs in $\Mod{A}.$
In Section 4, we develop a systematic method for constructing abelian model structures on the category of differential modules by using the induced cotorsion pairs in Section 2.

\section{Preliminaries}\label{property}
\noindent
Throughout this paper,
we assume that $\kk$ is a commutative ring and $A$ is an associative $\kk$-algebra, and assume that $\Q$ is a small $\kk$-linear category.
We denote by $\Mod{A}$ the category of left $A$-modules and by $_{A}\mathsf{Prj}$ the subcategory of $\Mod{A}$ consisting of projective left $A$-modules.
\begin{bfhpg}[\bf Cotorsion pairs]\label{cotpair}
Let $\mathcal{A}$ be an abelian category.
If $\X$ and $\Y$ are class of objects of $\mathcal{A}$, then we write
$$\quad \ \ {\X}^{\bot}=\{N\in {\mathcal{A}}\ |\ {\rm Ext}_{\A}^{1}(X,N)=0 {\rm \ for\ all}\ X\in\X\},\ \text{and}$$
$$^{\bot}{\Y}=\{M\in {\mathcal{A}}\ |\ {\rm Ext}_{\A}^{1}(M,Y)=0 {\rm \ for\ all}\ Y\in\Y\}.$$
A pair $(\X, \Y)$ of subcategories of $\A$ is called a \emph{cotorsion pair} if $\X^{\bot}=\Y$ and $^{\bot}\Y=\X$.

Recall from e.g. Enochs and Jenda \cite{rha} that a cotorsion pair $(\X,\Y)$ is \emph{complete} if for each object $M$ in $\A$ there is an exact sequence
$$0 \to M \to Y \to X \to 0 \quad \text{and}\quad 0 \to Y' \to X' \to M \to 0$$ with $X, X'\in\X$ and $Y,Y'\in\Y$. A cotorsion pair $(\X,\Y)$ is called \emph{hereditary} if $\X$ is closed under kernels of epimorphisms and $\Y$ is closed under cokernels of monomorphisms.

Let $(\X,\Y)$ be a cotorsion pair, and let $\mathsf{S}$ be a class of objects of $\A$. If $\mathsf{S}^{\bot}=\Y$, we say that $(\X,\Y)$ is generated by $\mathsf{S}$.
If $^{\bot}{\mathsf{S}}=\X$, we say that $(\X,\Y)$ is cogenerated by $\mathsf{S}$.
\end{bfhpg}

The following notion of projective/injective cotorsion pairs can be found in \cite[Definition 3.4]{Gil16}.

\begin{dfn}\label{df:projective cotorsion pair}
Let $\A$ be an abelian category with enough projectives and injectives. A complete cotorsion pair $(\X, \Y)$ is said to be \emph{projective} if $\Y$ is thick (that is, if two out of three of the terms in a short exact sequence are in $\Y$, then so is the third) and $\X\cap\Y$ coincides with the class of projective objects. The notion of \emph{injective} cotorsion pairs is defined dually.
\end{dfn}

\begin{rmk}\label{rmk:projective model structure}
A projective cotorsion pair $(\C, \W)$ determines a projective model structure on $\A$ such that $\C$ is the class of cofibrant objects, $\W$ is the class of trivial objects, and all objects in $\A$ are fibrant objects. Dually, an injective cotorsion pair $(\W, \FF)$ determines an injective model structure on $\A$ such that $\FF$ is the class of fibrant objects, $\W$ is the class of trivial objects, and all objects in $\A$ are cofibrant objects.
\end{rmk}

\begin{dfn}
Recall from \cite[Section 1]{G15} that the cotorsion pairs $(\mathcal{U},\mathcal{V})$ and $(\mathcal{X},\mathcal{Y})$ in abelian category $\mathcal{A}$ are \emph{compatible} if
 $\Ext[\Q,A]{1}{\mathcal{X}}{\mathcal{V}}=0$ and
 $\mathcal{X}\cap\mathcal{Y}=\mathcal{U}\cap\mathcal{V}$.
\end{dfn}

\begin{rmk}
Let $(\mathcal{U},\mathcal{V})$ and $(\mathcal{X},\mathcal{Y})$  be  compatible, complete and hereditary cotorsion pairs in an abelian category $\mathcal{A}$. Then by \cite[Theorem 1.1]{G15}, there exists an abelian model structure such that $\mathcal{U}$ is the class of cofibrant objects, $\mathcal{Y}$ is the class of fibrant objects, and $\mathcal{W}$ is the class of trivial objects, where
\begin{align*}
\mathcal{W} &= \Big\{M\in \mathcal{A} \mid
\exists\ s.e.s\
 0\to V\to X\to M\to 0
 ~~\text{with } X\in \mathcal{X} ~~\text{and}~~V\in \mathcal{V} \Big\},\\
 &=\Big\{M\in \mathcal{A} \mid
\exists\ s.e.s\
0\rightarrow M\rightarrow V'\rightarrow X'\rightarrow 0
~~\text{with } X'\in \mathcal{X} ~~\text{and}~~V'\in \mathcal{V} \Big\}.
\end{align*}
\end{rmk}

\begin{bfhpg}[\bf The category of $\Q$-shaped modules]\label{Q-shaped modules}
 We consider the category $\Mod{\Q,A}$ of $\kk$-linear functors from $\Q$ to the
 category $\Mod{A}$ of left $A$-modules, which will be referred to as {\bf category of $\Q$-shaped $A$-modules}.
 Let ${_{\Q,A}\mathsf{Prj}}$ denote the full subcategory consisting of all projective objects in $\Mod{\Q,A}$. When $A=\kk$, we write $\Mod{\Q}$ instead of $\Mod{\Q,\kk}$. Further symbols along with references to where they appear in \cite{HJ21,HJ23} will follow.
\end{bfhpg}

\begin{ipg}\label{HLSR}
Consider the following setup from  \cite[Setup 2.5]{HJ21}:
\begin{enumerate}
\item Hom-finiteness: Each hom $\kk$-module $\Q(p,q)$ is finitely generated and projective.
\item  Local boundedness: For any object $q$ in $\Q$, there are finitely many objects $p$ in $\Q$ such that $\Q(q, p) \neq 0$ and $\Q(p, q) \neq 0$.
\item Existence of a Serre functor: There exists a Serre functor, that is a $\kk$-linear auto-equivalence $\mathbb{S}: \Q \to \Q$ such that $\Q(p,q)\is\Hom[\kk]{\Q(q,\mathbb{S}(p))}{\kk}.$
\item Strong Retraction Property: For any object $q$ in $\Q$, the unit map $\kk\to\Q(q,q)$ (via $x\mapsto x\cdot\id_q$) has a $\kk$-module retraction; whence there is a $\kk$-module decomposition  $$\Q(q,q)=(\Bbbk\cdot{\rm id}_q)\oplus\mathfrak{r}_q$$ for $q\in \Q$ such that
\begin{itemize}
\item  $\mathfrak{r}_q\circ\mathfrak{r}_q\subseteq\mathfrak{r}_q$ for all $q$, and
\item $\Q(q,p)\circ \Q(p,q)\subseteq\mathfrak{r}_p$ for all $p\neq q$.
\end{itemize}
\end{enumerate}
\end{ipg}

\begin{ipg}\label{adjoint triple2}
Assume that $\Q$ satisfies the strong retraction property.
Then $\Q$ admits a pseudo-radical $\mathfrak{r}.$
Using the pseudo-radical $\mathfrak{r}$, one can define stalk functors:
$$S\langle q\rangle=\Q(q,-)/\mathfrak{r}(q,-)\in \Mod{\Q} \quad \text{and}\quad S\{ q\}=\Q(-,q)/\mathfrak{r}(-,q)\in \rMod{\Q},$$
see \cite[Lemma 7.7 and Definition 7.9]{HJ21}.
Recall from \cite[Proposition 7.15]{HJ21} that for every object $q$ in $\Q$, there is an adjoint triple $(C_{q},S_{q},K_{q})$ as follows:
  \begin{equation*}
  \xymatrix@C=4pc{
    \Mod{A}
    \ar[r]^-{S_{q}}
    &
    \Mod{\Q,A}
    \ar@/_1.8pc/[l]_-{C_{q}}
    \ar@/^1.8pc/[l]^-{K_{q}}
  }
  \qquad \text{given by} \qquad
  {\setlength\arraycolsep{1.5pt}
   \renewcommand{\arraystretch}{1.2}
  \begin{array}{rcl}
  C_q(X) &=& S\{q\} \otimes_\Q X \\
  S_q(M) &=& S\langle q\rangle\otimes_\kk M \\
  K_{q}(X) &=& \Hom[\Q]{S\langle q\rangle}{X}\;
  \end{array}
  }
  \end{equation*}
for $M\in \Mod{A}$ and $X\in \Mod{\Q,A}$.

  Using the stalk functors, one can for each $q\in\Q$ and $i\geq 0$ define (co)homology functors:
$$\mathbb{H}_{[q]}^{i}=\Ext[\Q]{i}{S\langle q\rangle}{-}\quad \text{and}\quad \mathbb{H}^{[q]}_{i}=\Tor[\Q]{i}{S\{q\}}{-},$$
which are functors from $\Mod{\Q,A}$ to $\Mod{A}$; see \cite[Definition 7.11]{HJ21}.
\end{ipg}

\begin{setup}\label{setup}
We now work with the following setup:
\begin{enumerate}
\item $\Q$ is a Hom-finite, local bounded $\kk$-preadditive category with a Serre functor and has the strong retraction property.
\item The pseudo-radical $\mathfrak{r}$ is nilpotent, that is, $\mathfrak{r}^{N}=0$ for some $N\in\mathbb{N}.$
\item The ring $\kk$ is noetherian and hereditary.
\end{enumerate}
\end{setup}

\begin{ipg}\label{adjoint triple}
 Recall from \cite[Corollary 3.9]{HJ21} that for every object $q$ in $\Q$ there is an adjoint triple $(F_{q},E_{q},G_{q})$ as follows:
  \begin{equation*}
  \xymatrix@C=4pc{
    \Mod{\Q,A}
    \ar[r]^-{E_{q}}
    &
    \Mod{A}
    \ar@/_1.8pc/[l]_-{F_{q}}
    \ar@/^1.8pc/[l]^-{G_{q}}
  }
  \qquad \text{given by} \qquad
  {\setlength\arraycolsep{1.5pt}
   \renewcommand{\arraystretch}{1.2}
  \begin{array}{rcl}
  F_q(M) &=& \Q(q,-) \otimes_\kk M \\
  E_q(X) &=& X\mspace{1.5mu}(q) \\
  G_{q}(M) &=& \Hom[\kk]{\Q(-, q)}{M}\;
  \end{array}
  }
  \end{equation*}
for $M\in \Mod{A}$ and $X\in \Mod{\Q,A}$.
Note that $E_{q}$ is the evaluation functor at $q.$
\end{ipg}

It is clear that the \textit{evaluation} functor $E_q$ is exact. Since $\Q(p,q)$ is a projective $\kk$-module for each pair of objects $p$ and $q$ in $\Q$ by the assumption, both the functors $F_q$ and $G_q$ are exact as well. Then by \cite[Lemma 5.1]{HJ19}, we have the following result.

\begin{lem}\label{n-iso}
Let $X$ be an object in $\Mod{\Q,A}$ and $q$ an object in $\Q$. Then for each $M \in \Mod{A}$ and $n \geqslant 0$, there are isomorphisms
$$\Ext[\Q,A]{n}{F_q(M)}{X}\is\Ext[A]{n}{M}{X(q)}$$
and
$$\Ext[A]{n}{X(q)}{M}\is\Ext[\Q,A]{n}{X}{G_q(M)}.$$
\end{lem}

\begin{bfhpg}[\bf Exact objects]\label{exact object}
Recall from \cite[Definition 4.1]{HJ21} that there is an equality
\begin{align*}
    \mathscr{E}
    &= \{X\in \Mod{\Q,A}\ |~ X^{\natural}~\text{has finite projective dimension}~\}\\
    &= \{X\in \Mod{\Q,A}\ |~ X^{\natural}~\text{has finite injective dimension}~\},
  \end{align*}
where $(-)^{\natural}$ is the forgetful functor; see \cite[Definition 3.2]{HJ21}.
Objects in $\mathscr{E}$ are called \emph{exact}.
This generalizes the exact complexes.
By \cite[Theorem 6.1]{HJ21}, $(^{\bot}\mathscr{E},\mathscr{E},\Mod{\Q,A})$ and
$(\Mod{\Q,A},\mathscr{E},\mathscr{E}^{\bot})$ are hereditary Hovey triples.
Recall from \cite{DN} that objects in $^{\bot}\mathscr{E}$ are called \emph{semi-projective}, and objects in $\mathscr{E}^{\bot}$ are called \emph{semi-injective}.
\end{bfhpg}

The following result is contained in the proof \cite[Theorem D]{HJ23}; but we present it as a lemma for the reader's convenience.

\begin{lem}\label{lem:exact object1}
Let $X\in \Mod{\Q,A}$.
Then $X$ is exact if and only if $$\Ext[\Q,A]{1}{S_{q}(A)}{X}=0$$ for each $q\in\Q.$
\end{lem}

\begin{lem}\label{lem:exact iso}
Let $Y$ be exact in $\Mod{\Q,A}$. Then the following isomorphisms hold for any $X\in \Mod{\Q,A}$ and  $q\in \Q$.
\begin{prt}
\item $\Ext[\Q,A]{1}{S_q(X)}{Y}\is\Ext[A]{1}{X}{K_{q}(Y)}$.
\item $\Ext[\Q,A]{1}{Y}{S_{q}(X)}\is\Ext[A]{1}{C_{q}(Y)}{X}.$
\end{prt}
\end{lem}
\begin{proof}
$(\mathrm{a}).$ Since $(S_{q},K_{q})$ is an adjoint pair and $S_{q}$ is exact, by \cite[Lemma 2.1]{ECO}, it suffices to prove that $K_{q}$ preserves exact sequences in $\Ext[\Q,A]{1}{S_{q}(X)}{Y}$.
Consider an exact sequence $$0\to Y\to M\to S_{q}(X)\to 0 $$ in $\Ext[\Q,A]{1}{S_{q}(X)}{Y}$.
We obtain the following exact sequence in $\Mod{A}$
$$0\to K_{q}(Y)\to K_{q}(M)\to K_{q}S_{q}(X)\to \Ext[\Q]{1}{S\langle q\rangle}{Y}.$$
Since $Y$ is  exact, we have $\Ext[\Q]{1}{S\langle q\rangle}{Y}=\mathbb{H}_{[q]}^{1}(Y)=0$ by \cite[Theorem 7.1]{HJ21}.
Hence we get the exact sequence $$0\to K_{q}(Y)\to K_{q}(M)\to K_{q}S_{q}(X)\to 0,$$
as desired.

$(\mathrm{b}).$ Since $(C_{q},S_{q})$ is an adjoint pair and $S_{q}$ is exact, by \cite[Lemma 2.1]{ECO}, it suffices to show that $C_{q}$ preserves exact sequences in $\Ext[\Q,A]{1}{Y}{S_{q}(X)}.$
Take an exact sequence
$$0\to S_{q}(X)\to N\to Y\to 0 $$ in $\Ext[\Q,A]{1}{Y}{S_{q}(X)}$.
Then we obtain the following exact sequence in $\Mod{A}$
$$L_{1}C_{q}(Y)\to C_{q}S_{q}(X)\to C_{q}(N)\to C_{q}(Y)\to 0.$$
Note that $C_{q}=S\{q\}\otimes_{\Q}-$ and $Y$ is exact.
Then $L_{1}C_{q}(Y)=\mathbb{H}_{1}^{[q]}(Y)=0$ by \cite[Theorem 7.1]{HJ21}.
Thus we get the following exact sequence
$$0\to C_{q}S_{q}(X)\to C_{q}(N)\to C_{q}(Y)\to 0,$$
as desired.
\end{proof}

\section{Induced cotorsion pairs in $\Mod{\Q,A}$}
In this section, we construct two cotorsion pairs in $\Mod{\Q,A}$ from a given cotorsion pair in $\Mod{A}.$
As a result, we establish a projective model structure on $\Mod{\Q,A}$ under mild conditions.
The next is the key definition in this paper.

\begin{dfn}\label{df:cotorsion pair}
Suppose that $(\mathcal{X},\mathcal{Y})$ is a cotorsion pair in $\Mod{A}$.
We introduce the following classes in $\Mod{\Q,A}$:
$$\quad \ \ \widetilde{\mathcal{Y}}= \{Y\in\Mod{\Q,A}\ |\ \Ext[\Q,A]{1}{S_{q}(C)}{Y}=0 ~\text{and} \ Y(q)\in \mathcal{Y} {\rm \ for~ any \ }C\in\mathcal{X}, ~q\in \Q \  \},$$
$$\quad \ \ \widetilde{\mathcal{X}}=\{X\in\Mod{\Q,A}\ |\ \Ext[\Q,A]{1}{X}{S_{q}(C)}=0\ \text{and}\ X(q)\in \mathcal{X} \text{ for any } C\in {\mathcal{Y}},~q\in \Q \  \}.$$
 We then set $\mathsf{dg} \mathcal{X}={^{\bot}\widetilde{\mathcal{Y}}}$ and $\mathsf{dg} \mathcal{Y}=\widetilde{\mathcal{X}}^{\bot}$.
\end{dfn}

\begin{rmk}
Consider the situation where $\Q$ is the mesh category of the repetitive quiver of $\overrightarrow{\mathbb{A}_{2}}=\bullet\to\bullet.$
Then $\Mod{\Q,A}$ is equivalent to the category $\mathrm{Ch}(A)$ of chain complexes of left $A$-modules (see \cite[Remark 8.5 and Example 8.13]{HJ21}), and $\widetilde{\mathcal{X}}$ (resp. $\widetilde{\mathcal{Y}}$) coincides with the class of $\mathcal{X}$ (resp. $\mathcal{Y}$) complexes introduced by Gillespie in \cite[Definition 3.3]{Gi04}; see Proposition \ref{prop:KC} below.
\end{rmk}

\begin{lem}\label{lem:contain contractible objects}
Suppose that $(\mathcal{X},\mathcal{Y})$ is a cotorsion pair in $\Mod{A}$. Then the following statements hold.
\begin{prt}
\item $\widetilde{\mathcal{Y}}$ contains $G_{q}(Y)$ on objects $Y\in \mathcal{Y}$ and  $q\in\Q$.
\item $\mathsf{dg} \mathcal{X}$ contains $F_{q}(X)$ and $S_{q}(X)$ on objects $X\in \mathcal{X}$ and  $q\in\Q$.
\item $\widetilde{\mathcal{X}}$ contains $F_{q}(X)$ on objects $X\in \mathcal{X}$ and  $q\in\Q$.
\item $\mathsf {dg} \mathcal{Y}$ contains $G_{q}(Y)$ and $S_{q}(Y)$ on objects $Y\in \mathcal{Y}$ and  $q\in\Q$.
\end{prt}
\end{lem}
\begin{prf*}
 $(\mathrm{a}).$ We mention that $G_{q}(Y)=\Hom[\kk]{\Q(-,q)}{Y}$ for any $Y\in \Y$ and $q\in \Q$.
 Thus $G_{q}(Y)(p)=\Hom[\kk]{\Q(p,q)}{Y}$ for any $p\in \Q$.
 Since $\Q(q,p)$ is finitely generated and projective as a $\kk$-module by Setup \ref{setup}, it follows that $G_{q}(Y)(p)$ is in $\mathcal{Y}$. Now it is enough to show that $\Ext[\Q,A]{1}{S_{p}(X)}{G_{q}(Y)}=0$ for any $X\in \X$ and $p\in\Q.$
For any $X\in \mathcal{X}$, we have
\begin{align*}
\Ext[\Q,A]{1}{S_{p}(X)}{G_{q}(Y)}&\cong \Ext[A]{1}{X}{K_{p}G_{q}(Y)}\\
&\cong 0,
\end{align*}
where the first isomorphism follows from Lemma \ref{lem:exact iso} as $G_{q}(Y)$ is exact by \cite[Lemma 7.14]{HJ21}, the second isomorphism holds by \cite[Lemma 7.28]{HJ21}.

$(\mathrm{b}).$ By the definition of $\mathsf{dg} \mathcal{X}$, we know that $S_{q}(X)\in \mathsf{dg} \mathcal{X}$ for any $X\in \X$ and $q\in\Q$. On the other hand,
given $Y\in \widetilde{\mathcal{Y}}$, we need to show that $\Ext[\Q,A]{1}{F_{q}(X)}{Y}=0$ for all $q\in\Q$ and $X\in \X$.
Indeed, by Lemma \ref{n-iso}, we have  $$\Ext[\Q,A]{1}{F_{q}(X)}{Y}\cong \Ext[A]{1}{X}{Y(q)}.$$ Since $Y(q)\in\Y$ by Definition \ref{df:cotorsion pair}, one has $\Ext[A]{1}{X}{Y(q)}=0$ as $(\X,\Y)$ is a cotorsion pair.
Consequently, $\Ext[\Q,A]{1}{F_{q}(X)}{Y}=0$, as desired.

The proofs of $(\mathrm{c})$ and $(\mathrm{d})$ are respectively dual to those of $(\mathrm{a})$ and $(\mathrm{b})$.
\end{prf*}

\begin{cor}\label{cor:dw}
Let $(\X,\Y)$ be a cotorsion pair in $\Mod{A}.$
Then the following statements hold.
\begin{prt}
\item If $X\in \DX$, then $X(q)\in \X$ for each $q\in \Q.$
\item If $Y\in \DY$, then $Y(q)\in \Y$ for each $q\in \Q.$
\end{prt}
\end{cor}
\begin{proof}
We only prove $(\mathrm{a})$ since the proof of $(\mathrm{b})$ is dual to that of $(\mathrm{a})$.

For any $M\in \Y$, we have $G_{q}(M)\in \WY$ for each $q\in \Q$ by Lemma \ref{lem:contain contractible objects}.
If $X\in \DX$, then $$0=\Ext[\Q,A]{1}{X}{G_{q}(M)}\cong \Ext[A]{1}{X(q)}{M}.$$
It implies that $X(q)\in \X$ for all $q\in\Q$ as $(\X,\Y)$ is a cotorsion pair.
\end{proof}

\begin{prp}\label{prop:cotorsion pair}
Let $(\mathcal{X},\mathcal{Y})$ be a cotorsion pair in $\Mod{A}$.
Then $(\mathsf{dg} \mathcal{X},\widetilde{\mathcal{Y}})$ and $(\widetilde{\mathcal{X}},\mathsf{dg} \mathcal{Y})$ are cotorsion pairs in $\Mod{\Q,A}$.
\end{prp}
\begin{proof}
By Definition \ref{df:cotorsion pair}, we already have $\mathsf{dg}\mathcal{X}={^{\bot}\widetilde{\mathcal{Y}}}$.
So we wish to show ${\mathsf{dg}\mathcal{X}}^{\bot}=\widetilde{\mathcal{Y}}.$
Note that $\widetilde{\mathcal{Y}}\subseteq {\mathsf{dg}\mathcal{X}}^{\bot}$ by the definition of $\mathsf{dg}\mathcal{X}.$
Thus we need to show that ${\mathsf{dg}\mathcal{X}}^{\bot}\subseteq\widetilde{\mathcal{Y}}.$
So take $Y\in \Mod{\Q,A}$ satisfying $\Ext[\Q,A]{1}{X}{Y}=0$ for all $X\in \mathsf{dg}\mathcal{X}$.
By Lemma \ref{lem:contain contractible objects}(b), we can take $X=F_{q}(M)$ for any $M\in \X$ and $q\in \Q.$
Then $$0=\Ext[\Q,A]{1}{F_{q}(M)}{Y} \cong \Ext[A]{1}{M}{Y(q)}.$$
It implies that $Y(q)\in \mathcal{Y}$ for every $q\in \Q.$
Next, by Lemma \ref{lem:contain contractible objects}(b) again, we may take $X=S_{q}(M)$ for any $M\in \mathcal{X}$ and $q\in\Q.$
Then $\Ext[\Q,A]{1}{S_{q}(M)}{Y}=0$, proving that $Y\in \widetilde{\mathcal{Y}}.$
Hence we complete the proof that $(\mathsf{dg}\mathcal{X},\widetilde{\mathcal{Y}})$ is also a cotorsion pair in $\Mod{\Q,A}$.

By a dual argument, we can show that $(\widetilde{\mathcal{X}},\mathsf{dg} \mathcal{Y})$ is a cotorsion pair in $\Mod{\Q,A}$.
\end{proof}

\begin{lem}\label{lem:exact object2}
Let $X\in \Mod{\Q,A}$ and $I$ be an injective module in $\Mod{A}$. Then for $i\geq 1$ we have the isomorphism $$\Ext[\Q,A]{i}{X}{S_{q}(I)}\cong \Hom[\kk]{\mathbb{H}_{i}^{[q]}(X)}{I}.$$
In particular, $X$ is exact if and only if $\Ext[\Q,A]{i}{X}{S_{q}(E)}=0$ for some injective cogenerator $E$ in $\Mod{A}.$
\end{lem}
\begin{proof}
Consider the projective resolution of $S\{q\}$ in $\rMod{\Q}$
$$P_{\bullet}=\cdots\to P_{2} \to P_{1}\to P_{0}\to 0.$$
By \cite[Lemma 3.8]{HJ21}, it is easy to check that each $\Hom[\kk]{P_{i}}{I}$ is injective in $\Mod{\Q,A}$, and then
$$\Hom[\kk]{P_{\bullet}}{I}=0\to \Hom[\kk]{P_{0}}{I}\to \Hom[\kk]{P_{1}}{I}\to \Hom[\kk]{P_{2}}{I}\to \cdots $$
is an injective resolution of $\Hom[\kk]{S\{q\}}{I}.$
It follows that we get the following isomorphisms:
\begin{align*}
\Ext[\Q,A]{i}{X}{S_{q}(I)}&\cong \Ext[\Q,A]{i}{X}{\Hom[\kk]{S\{q\}}{I}}\\
&=\mathrm{H}_{-i}\Hom[\Q,A]{X}{\Hom[\kk]{P_{\bullet}}{I}}\\
&\cong \mathrm{H}_{-i}\Hom[A]{P_{\bullet}\otimes_{\Q}X}{I}\\
&\cong \Hom[A]{\mathrm{H}_{i}(P_{\bullet}\otimes_{\Q}X)}{I}\\
&=\Hom[A]{\mathrm{Tor}_{i}^{\Q}(S\{ q\},X)}{I}\\
&=\Hom[\kk]{\mathbb{H}_{i}^{[q]}(X)}{I}.
\end{align*}
In this computation, the first isomorphism holds by the fact that $S_{q}\cong \Hom[\kk]{S\{q\}}{-}$; see the proof of \cite[Proposition 7.15]{HJ21}.
The second isomorphism is induced by \cite[Proposition 3.8]{HJ21}.
The third isomorphism holds as $I$ is injective in $\Mod{A}.$
\end{proof}

\begin{prp}\label{prop:KC}
Let $(\mathcal{X},\mathcal{Y})$ be a cotorsion pair on $\Mod{A}$, and $X$ be in $\Mod{\Q,A}$. Then the following statements hold.
\begin{prt}
\item Then $X\in \widetilde{\mathcal{X}}$ if and only if $X$ is exact, $C_{q}(X)\in \mathcal{X}$ and $X(q)\in \mathcal{X}$ for every $q\in\Q.$
\item Then $Y\in \widetilde{\mathcal{Y}}$ if and only if $Y$ is exact, $K_{q}(Y)\in \mathcal{Y}$ and $Y(q)\in \mathcal{Y}$ for every $q\in\Q.$
\end{prt}
\end{prp}
\begin{proof}
$(\mathrm{a}).$ If $X\in \widetilde{\mathcal{X}}$, then $\Ext[\Q,A]{1}{X}{S_{q}(Y)}=0$ and $X(q)\in \mathcal{X}$ for every $q\in \Q$ and $Y\in \mathcal{Y}.$
Let $E$ be an injective cogenerator in $\Mod{A}.$
Then $E\in \mathcal{Y}$, and then $\Ext[\Q,A]{1}{X}{S_{q}(E)}=0$.
It follows from Lemma \ref{lem:exact object2} that $X$ is exact.
 According to Lemma \ref{lem:exact iso}, one has $$0=\Ext[\Q,A]{1}{X}{S_{q}(Y)}\is\Ext[A]{1}{C_{q}(X)}{Y},$$
implying that $C_{q}(X)$ is in $\mathcal{X}.$
Conversely, if $X$ is exact, $X(q)$ and $C_{q}(X)$ are in $\mathcal{X}$ for all $q\in \Q$, then it follows from Lemma \ref{lem:exact iso} that $$\Ext[\Q,A]{1}{X}{S_{q}(Y)}\is\Ext[A]{1}{C_{q}(X)}{Y}=0$$ for any $Y\in \mathcal{Y}$.
Therefore, $X\in \widetilde{\mathcal{X}}$ by Definition \ref{df:cotorsion pair}.

$(\mathrm{b}).$ If $Y\in \widetilde{\mathcal{Y}},$ then $\Ext[\Q,A]{1}{S_{q}(X)}{Y}=0$ and $Y(q)\in \mathcal{Y}$ for every $q\in\Q$ and $X\in \mathcal{X}.$
Note that $A\in \mathcal{X}$. Then $\Ext[\Q,A]{1}{S_{q}(A)}{Y}=0$.
We conclude that $Y$ is exact by Lemma \ref{lem:exact object1}.
Hence it follows from Lemma \ref{lem:exact iso} that $$0=\Ext[\Q,A]{1}{S_{q}(X)}{Y}=\Ext[\Q,A]{1}{X}{K_{q}(Y)}.$$
So $K_{q}(Y)$ is in $\mathcal{Y}$.
Conversely, if $Y$ is exact, $Y(q)$ and $K_{q}(Y)$ are in $\mathcal{Y}$ for each $q\in \Q$, then, by Lemma \ref{lem:exact iso}, $$\Ext[\Q,A]{1}{S_{q}(X)}{Y}\cong \Ext[A]{1}{X}{K_{q}(Y)}=0$$
for any $X\in \mathcal{X}.$
Hence $Y\in \widetilde{\mathcal{Y}}$ by Definition \ref{df:cotorsion pair}.
\end{proof}

\begin{cor}\label{cor:hereditary cotorsion pair}
 Let $(\mathcal{X},\mathcal{Y})$ be a cotorsion pair in $\Mod{A}$, and let $(\mathsf{dg} \mathcal{X},\widetilde{\mathcal{Y}})$ and $(\widetilde{\mathcal{X}},\mathsf{dg} \mathcal{Y})$ be the induced cotorsion pairs in $\Mod{\Q,A}$. Then the following statements are equivalent.
\begin{eqc}
\item $(\mathcal{X},\mathcal{Y})$ is hereditary.
\item $(\widetilde{\mathcal{X}},\mathsf{dg} \mathcal{Y})$ is hereditary.
\item $(\mathsf{dg} \mathcal{X},\widetilde{\mathcal{Y}})$ is hereditary.
\end{eqc}
\end{cor}
\begin{proof}
\proofofimp{i}{ii} Let $0\to X'\to X\to X''\to 0$ be an exact sequence in $\Mod{\Q,A}$ with $X,X''\in \widetilde{\mathcal{X}}.$
Then it follows from Proposition \ref{prop:KC} that $X$ and $X''$ are exact object, and then so is $X'$ by \cite[Theorem 4.1]{HJ21}.
For each $q\in \Q,$ we have the following exact sequence
$$L_{1}C_{q}(X'')\to C_{q}(X')\to C_{q}(X)\to C_{q}(X'')\to 0.$$
Note that $C_{q}=S\{q\}\otimes_{\Q}-$.
Then $L_{1}C_{q}(X'')=\mathbb{H}_{1}^{[q]}(X'')=0$ by \cite[Theorem 7.1]{HJ21}.
Hence we have the following exact sequence
$$0\to C_{q}(X')\to C_{q}(X)\to C_{q}(X'')\to 0.$$
Since $X$ and $X''$ are in $\widetilde{\mathcal{X}}$, it follows from Lemma \ref{prop:KC} that $C_{q}(X)$ and $C_{q}(X'')$ are in $\mathcal{X}$.
So we have $C_{q}(X')\in \mathcal{X}$ as $(\X,\Y)$ is hereditary. By Lemma \ref{prop:KC} again, $X'\in \widetilde{\mathcal{X}}.$
Hence $(\widetilde{\mathcal{X}},\mathsf{dg} \mathcal{Y})$ is hereditary.

\proofofimp{ii}{i} Let $0\to Y_{1}\to Y_{2}\to Y_{3}\to 0$ be an exact sequence in $\Mod{A}$ with $Y_{1},Y_{2}\in \mathcal{Y}.$
Then for each $q\in \Q,$ the sequence
$$0\to S_{q}(Y_{1})\to S_{q}(Y_{2})\to S_{q}(Y_{3})\to 0$$
is exact, and $S_{q}(Y_{1})$ and $S_{q}(Y_{2})$ are in $\mathsf{dg}\mathcal{Y}$ by Lemma \ref{lem:contain contractible objects}$(\mathrm{d})$.
Hence $S_{q}(Y_{3})\in \mathsf{dg}\mathcal{Y}.$
For any $X\in \mathcal{X}$, we have $F_{q}(X)\in \widetilde{\mathcal{X}}$ by Lemma \ref{lem:contain contractible objects}$(\mathrm{c})$.
Consequently, it follows from Lemma \ref{n-iso} that $$0=\Ext[\Q,A]{1}{F_{q}(X)}{S_{q}(Y_{3})}\cong \Ext[A]{1}{X}{Y_{3}},$$
proving $Y_{3}\in \mathcal{Y}$, as desired.

The $(i) \Longleftrightarrow (iii)$ can be proved by a dual argument as above.
\end{proof}

Recall from \cite[Definition 5.13]{HJ23} that a \emph{cycle} in $\Q$ is a sequence of morphisms in $\Q$
$$q_{1}\xrightarrow{g_{1}}q_{2}\xrightarrow{g_{2}}\cdots\xrightarrow{g_{n-1}}q_{n}\xrightarrow{g_{n}}q_{n+1},$$
with $q_{1}=q_{n+1}$, where each $g_{i}$ is non-zero and belongs to the pseudo-radical $\mathfrak{r}$ from \ref{adjoint triple2}.
By \cite[Definition 5.1]{HJ23},
the support of an object $X$ in $\Mod{\Q,A}$ is defined as
$$\mathrm{supp} X=\{q\in \Q \mid X(q)\neq 0\}.$$ We say that $X$ has finite support if $\mathrm{supp} X$ is a finite set.

For a class $\sf S$ of left $A$-modules, we denote by $\mathrm{add} (\sf S)$ the full subcategory consisting of direct summands of finite direct sums of modules in $\sf S$.
\begin{lem}\label{lem:support}
Let $\sf S$ be a class of $A$-modules and $X\in \Mod{\Q,A}$.
Assume that $\Q$ has no cycles.
If $\{S_{q}(C)\ | \ C\in \mathsf{S} ~\text{and}~ q\in \Q \}\subseteq {^{\bot}X}$, then
$$\{Z\in \Mod{\Q,A}\ | \ Z~\text{has finite support and}\ Z(q)\in \mathrm{add}(\mathsf{S})~ \text{for each}\ q\in \Q \}\subseteq {^{\bot}X}.$$
\end{lem}
\begin{proof}
 Let $Z\in \Mod{\Q,A}$ such that $Z$ has finite support and $Z(q)\in \mathrm{add}(\sf S)$ for each $q\in \Q.$
We argue by induction on the cardinality $n=|\mathrm{supp}Z|$ of the support of $Z.$

If $n=0$, then $\mathrm{supp}Z=\emptyset$, and so $Z=0\in {^{\bot}X}.$
Assume that $n>0$.
By \cite[Lemma 5.14]{HJ23}, there exists $q\in\Q$ such that $K_{q}(Z)=Z(q).$
Let $\varepsilon_{q}$ be the counit of the adjunction $(S_{q},K_{q})$ from \ref{adjoint triple2}.
Note that the morphism $$S_{q}(Z(q))=S_{q}K_{q}(Z)\xrightarrow{\varepsilon_{q}^{Z}}Z$$
is monic since $\varepsilon_{q}^{Z}(q):Z(q)\to Z(q)$ is the identity map and $\varepsilon_{q}^{Z}(p):0\to Z(p)$ is the zero map for $p\neq q.$
Hence one has an exact sequence in $\Mod{\Q,A}$
\begin{equation*}\label{diag01}
\tag{\ref{lem:support}.1}
0\to S_{q}(Z(q))\to Z\to Y\to 0,
\end{equation*}
where $Y(q)=0$ and $Y(p)=Z(p)$ for $p\neq q.$
It follows that $Y(q)\in \mathrm{add}(\sf S)$ and $|\mathrm{supp}Y|=n-1.$
By induction hypothesis, we conclude that $Y\in {^{\bot}X}.$
Application of the functor $\Hom[\Q,A]{-}{X}$ to the sequence (\ref{diag01}) yields an exact sequence
$$\Ext[\Q,A]{1}{Y}{X}\to \Ext[\Q,A]{1}{Z}{X}\to \Ext[\Q,A]{1}{S_{q}(Z(q))}{X}.$$
Since $\Ext[\Q,A]{1}{Y}{X}=0=\Ext[\Q,A]{1}{S_{q}(Z(q))}{X}$, it follows that $Z\in {^{\bot}X}$.
Thus we complete the proof of the lemma.
\end{proof}

\begin{thm}\label{thm:complete cotorsion pair}
Let $(\mathcal{X},\mathcal{Y})$ be the complete cotorsion pair generated by a set $\sf S$ containing a generator $U$ in $\Mod{A}$.
Assume that $\Q$ has no cycles.
Then $(\mathsf{dg}\mathcal{X},\widetilde{\mathcal{Y}})$ is a complete cotorsion pair in $\Mod{\Q,A}$, and it is generated by the set $\{S_{q}(C)\ | \ C\in \mathsf{S}, q\in \Q \}.$
\end{thm}
\begin{prf*}
By \cite[Corollary 2.15(3)]{SaorinStovicek}, we only need to show
$$\widetilde{\mathcal{Y}}= \{S_{q}(C)\ | \ C\in \mathsf{S}, q\in \Q \}^{\bot},$$
implying that $(\mathsf{dg} \mathcal{X},\widetilde{\mathcal{Y}})$ is a complete cotorsion pair in $\Mod{\Q,A}.$
We mention that $$\widetilde{\mathcal{Y}}\subseteq \{S_{q}(C)\ | \ C\in \mathsf{S}, q\in \Q \}^{\bot}$$
is clear.
We only need to prove that $\{S_{q}(C)\ | \ C\in \mathsf{S}, q\in \Q \}^{\bot}\subseteq\widetilde{\mathcal{Y}}.$
For each $q\in \Q$ and $C\in \sf S$, by \cite[Proposition 5.7]{HJ23}, we have the following objectwise split exact sequence
\begin{equation*}\label{diag02}
\tag{\ref{thm:complete cotorsion pair}.1}
0\to S_{q}(C)\to \mathbb{G}(S_{q}(C))\to \mathbb{C}(S_{q}(C))\to 0.
\end{equation*}
Since $\mathrm{supp}S_{q}(C)=q,$ it follows from \cite[Lemma 5.8]{HJ23} that both $\mathbb{G}(S_{q}(C))$ and $\mathbb{C}(S_{q}(C))$ have finite support.
Note that $$\mathbb{G}(S_{q}(C))(p)=G_{q}(C)(p)=\Hom[\kk]{\Q(p,q)}{C}.$$
It follows that
$\mathbb{G}(S_{q}(C))(p)\in \mathrm{add}(\sf S)$ for each $p\in \Q$
as $\Q(p,q)$ is a finitely generated and projective as a $\kk$-module.
For any $Y\in \{S_{q}(C)\ | \ C\in \mathsf{S}, q\in \Q \}^{\bot}$, it follows from Lemma \ref{lem:support} that $\Ext[\Q,A]{1}{\mathbb{G}(S_{q}(C))}{Y}=0.$
On the other hand, we have the following isomorphisms
\begin{align*}
\Ext[\Q,A]{1}{\mathbb{G}(S_{q}(C))}{Y}&=\Ext[\Q,A]{1}{G_{q}(C)}{Y}\\
&\cong\Ext[\Q,A]{1}{F_{\mathbb{S}^{-1}q}(C)}{Y}\\
&\cong\Ext[A]{1}{C}{Y(\mathbb{S}^{-1}q)}
\end{align*}
for any $q\in \Q$ and $C\in \sf S$, where the first isomorphism holds by \cite[Lemma 3.4]{HJ23}, and the second isomorphism follows from Lemma \ref{n-iso}.
Hence $\Ext[A]{1}{C}{Y(\mathbb{S}^{-1}q)}=0,$ and so
 it yields that $Y(q)\in \mathcal{Y}$
as $\mathbb{S}:\Q\to \Q$ is an auto-equivalence.
Since $Y\in \{S_{q}(C)\ | \ C\in \mathsf{S}, q\in \Q \}^{\bot}$ and $\sf S$ contains a generator $U$, it follows from Lemma \ref{lem:exact object1} that $Y$ is exact.
Hence for any $C\in \sf S$ and $q\in \Q,$ we have
$$0=\Ext[\Q,A]{1}{S_{q}(C)}{Y}\cong \Ext[A]{1}{C}{K_{q}(Y)}$$
by Lemma \ref{lem:exact iso},
implying that $K_{q}(Y)\in \mathcal{Y}.$
Hence we conclude that $Y\in\WY$ by Lemma \ref{prop:KC}$(\mathrm{b})$.
\end{prf*}

\begin{prp}\label{thm:projective cotorsion pair}
Let $(\mathcal{X},\mathcal{Y})$ be a projective cotorsion pair generated by a set $\sf S$ containing a  generator $U$ in $\Mod{A}$. Assume that $\Q$ has no cycles.
Then $(\sf dg \mathcal{X},\widetilde{\mathcal{Y}})$ is a projective cotorsion pair in $\Mod{\Q,A}.$
\end{prp}
\begin{proof}
By Theorem \ref{thm:complete cotorsion pair}, we know that $(\sf dg \mathcal{X},\widetilde{\mathcal{Y}})$ is a complete cotorsion pair in $\Mod{\Q,A}$.
Now we first show that $\widetilde{\mathcal{Y}}$ is thick.
Consider the exact sequence in $\Mod{\Q,A}$
\begin{equation*}\label{diag03}
\tag{\ref{thm:projective cotorsion pair}.1}
0\to Y_{1}\to Y_{2}\to Y_{3}\to 0.
\end{equation*}
If $Y_{2}$ and $Y_{3}$ are in $\widetilde{\mathcal{Y}}$, then $Y_{2}$ and $Y_{3}$ are exact by Proposition \ref{prop:KC}. It follows from \cite[Theorem 4.4]{HJ21} that $Y_{1}$ is exact.
Applying the functor $K_{q}$ to (\ref{diag03}), we obtain the following exact sequence
$$0\to K_{q}(Y_{1})\to K_{q}(Y_{2})\to K_{q}(Y_{3})\to R_{1}K_{q}(Y_{1}).$$
 Note from \cite[Theorem 7.1]{HJ21} that $R_{1}K_{q}(Y_{1})=\mathbb{H}_{[q]}^{1}(Y_{1})=0$ as $Y_{1}$ is exact.
 Thus we obtain the exact sequence $$0\to K_{q}(Y_{1})\to K_{q}(Y_{2})\to K_{q}(Y_{3})\to 0.$$
 Since the sequence (\ref{thm:projective cotorsion pair}.1) is exact, it implies that
 $$0\to Y_{1}(q)\to Y_{2}(q)\to Y_{3}(q)\to 0$$
 is exact for every $q\in\Q$.
 By Lemma \ref{prop:KC}($\mathrm{b}$), we know that $K_{q}(Y_{2}), Y_{2}(q)$,  $K_{q}(Y_{3})$ and $Y_{3}(q)$ are in $\Y.$
 Since $\mathcal{Y}$ is thick, it follows that both $K_{q}(Y_{1})$ and $Y_{1}(q)$ are in $\mathcal{Y}$.
 By Lemma \ref{prop:KC} again, $Y_{1}\in \widetilde{\mathcal{Y}}.$
 Similarly, if $Y_{1}$ and $Y_{2}$ are in $\widetilde{\mathcal{Y}}$, we can prove that $Y_{3}\in \widetilde{\mathcal{Y}}$.
 Note that $\widetilde{\mathcal{Y}}$ is closed under extensions as $(\mathsf{dg} \mathcal{X},\widetilde{\mathcal{Y}})$ is a  cotorsion pair.
 Therefore, $\widetilde{\mathcal{Y}}$ is  thick.
 Next we show that $\mathsf{dg}\mathcal{X}\cap\widetilde{\mathcal{ Y}}={_{\Q,A}\mathsf{Prj}}.$

$``\supseteq":$ Evidently, ${_{\Q,A}\mathsf{Prj}}\subseteq\mathsf{dg}\mathcal{X}.$
Now we show that ${_{\Q,A}\mathsf{Prj}}\subseteq\WY.$
Let $X\in {_{\Q,A}\mathsf{Prj}}$. Then, by the proof of \cite[Theorem 7.29]{HJ21}, we know that $X$ is exact and $X\cong \bigoplus_{q\in\Q}F_{q}C_{q}(X)$ with $C_{q}(X)\in {_{A}\mathsf{Prj}}$ for every $q\in\Q.$
For any $p\in\Q,$ we have
$$X(p)\cong \bigoplus_{q\in\Q} F_{q}C_{q}(X)(p)\cong \bigoplus_{q\in\Q}\Q(q,p)\otimes_{\kk}C_{q}(X).$$
Since $\Q(q,p)$ is a finitely generated and projective as a $\kk$-module, it follows that $X(p)\in {_{A}\mathsf{Prj}}\subseteq\mathcal{Y}$ as $(\mathcal{X},\mathcal{Y})$ is a projective cotorsion pair.
On the other hand, for any $p\in \Q$, one has the following isomorphisms of $A$-modules
\begin{align*}
K_{p}(X)&=K_{p}(\bigoplus_{q\in\Q} F_{q}C_{q}(X))\\
&\cong \bigoplus_{q\in\Q} K_{p}F_{q}C_{q}(X)\\
&\cong \bigoplus_{q\in\Q} K_{p}G_{\mathbb{S}(q)}C_{q}(X)\\
&=C_{\mathbb{S}^{-1}(p)}(X),
\end{align*}
where the first isomorphism holds since $K_{q}$ preserve coproducts by \cite[Proposition 4.6]{HJ23}, the second isomorphism follows from \cite[Lemma 3.4]{HJ23}, and the last equality holds by \cite[Lemma 7.28(b)]{HJ21}.
Thus $K_{p}(X)\in {_{A}\mathsf{Prj}}\subseteq\mathcal{Y}$ for every $p\in\Q.$
Consequently, by Proposition \ref{prop:KC}, $X\in \WY$ and then ${_{\Q,A}\mathsf{Prj}}\subseteq\WY.$

 $``\subseteq":$ Assume that $X$ belongs to $\mathsf{dg}\mathcal{X}\cap\widetilde{\mathcal{Y}}.$
 Since $\Mod{\Q,A}$ has enough projectives by \cite[Proposition 3.12]{HJ21}, we
 have an exact sequence $0\to Y\to P\to X\to 0$ in $\Mod{\Q,A}$ with $P$ projective.
 Note that $P\in \widetilde{\mathcal{Y}}$ by the above argument.
 It follows that $Y\in \widetilde{\mathcal{Y}}$ as $\widetilde{\mathcal{Y}}$ is thick.
 We also have $X\in \mathsf{dg}\mathcal{X},$ implying that $\Ext[\Q,A]{1}{X}{Y}=0$ and the sequence $0\to Y\to P\to X\to 0$ splits.
 Hence $X$ is a direct summand of the projective object $P$, and then $X\in {_{\Q,A}\mathsf{Prj}}.$
\end{proof}

\begin{thm}\label{thm:projective model structure}
Let $(\mathcal{X},\mathcal{Y})$ be a projective cotorsion pair generated by a set $\sf S$ containing a generator $U$ in $\Mod{A}$.
Assume that $\Q$ has no cycles.
Then there exists an abelian model structure on $\Mod{\Q,A}$ such that $\mathsf{dg}\mathcal{X}$ is the subcategory of cofibrant objects, $\widetilde{\mathcal{Y}}$ is the subcategory of trivial objects, and every object is fibrant.
Its homotopy  category is equivalent to the stable category $\mathsf{dg}\mathcal{X}/_{\Q,A}\mathsf{Prj}$.
\end{thm}
\begin{proof}
We claim that $(\mathsf{dg}\mathcal{X},\widetilde{\mathcal{Y}},\Mod{\Q,A})$ is a Hovey triple in $\Mod{\Q,A}.$
\begin{itemize}
\item The subcategory $\widetilde{\mathcal{Y}}$ is thick by Proposition \ref{thm:projective cotorsion pair}.
\item We know from Proposition \ref{thm:projective cotorsion pair} that $(\mathsf{dg}\mathcal{X}\cap\widetilde{\mathcal{Y}},\Mod{\Q,A})=(_{\Q,A}\mathsf{Prj},\Mod{\Q,A})$.
    Since $\Mod{\Q,A}$ has enough projectives by \cite[Proposition 3.2(a)]{HJ21}, it follows that $(\mathsf{dg}\mathcal{X}\cap\widetilde{\mathcal{Y}},\Mod{\Q,A})=(_{\Q,A}\mathsf{Prj},\Mod{\Q,A})$ is a complete cotorsion pair.
\item Clearly, $(\mathsf{dg}\mathcal{X},\widetilde{\mathcal{Y}}\cap\Mod{\Q,A})=(\mathsf{dg}\mathcal{X},\widetilde{\mathcal{Y}})$. This is a complete cotorsion pair by Theorem \ref{thm:complete cotorsion pair}.
\end{itemize}
Consequently, the desired conclusion follows from \cite[Theorem 2.2]{Ho02} and \cite[Theorem 4.3]{G16}.
\end{proof}

\begin{rmk}
Set $(\mathcal{X},\mathcal{Y})=(_{A}\mathsf{Prj},\Mod{A})$ in Theorem \ref{thm:projective model structure}.
Then $\widetilde{\mathcal{Y}}$ is the class of exact objects, $\mathsf{dg}\mathcal{X}$ is the class of semi-projective objects, and the corresponding homotopy category is the $\Q$-shaped derived category introduced by Holm and J{\o}rgensen in \cite{HJ21}.
\end{rmk}

\section{Model structures on the category of differential modules}
In this section, we will develop a systematic method for constructing abelian model structures on the category of differential modules using the induced cotorsion pairs in Section 2.
We begin by recalling the definition of differential modules.
\begin{bfhpg}[\bf Differential modules]
Recall from \cite{LLAHBF2007} that a differential left $A$-module is a left $A$-module $X$ equipped with an endomorphism $d_{X}:X\to X$, called the differentiation of $X$, satisfying $d_{X}d_{X}=0.$ Sometimes we say this pair $(X,d_{X})$ is a differential left $A$-module.

A morphism from a differential left $A$-module $(X,d_{X})$ to a differential left $A$-module $(Y,d_{Y})$ is a morphism $f:X\rightarrow Y$
such that the following diagram commutes:
$$\xymatrix{
  X \ar[d]_{f} \ar[r]^{d_{X}}& X \ar[d]^{f} \\
  Y \ar[r]^{d_{Y}} & Y  }$$
We denote by $\sf{Dif} A$ the category of all differential left $A$-modules.

\begin{rmk}\label{rm:Jordan}
Consider the Jordan quiver; it has a single vertex $\ast$ and a single loop $\varepsilon$ with the relation $\varepsilon^{2}=0$.
\vspace{-2mm}
\[
\begin{tikzpicture}[baseline=-0.5ex]
\node (v) at (0,0) {$*$};
\draw[->]
(v) to[out=35,in=-35,looseness=7]
node[right] {$\varepsilon$} (v);
\end{tikzpicture}
\]
Let $\Q$ be the path category of the quiver modulo the ideal generated by $(\varepsilon^{2})$. Then $\Mod{\Q,A}$ can be identified with $\Di$.
By \cite[A.2]{HJ23}, $\Q$ satisfies all the assumptions in \ref{HLSR} and the pseudo-radical $\mathfrak{r}=(\varepsilon)$ is nilpotent with $\mathfrak{r}^{2}=0.$
Hence we are in the setting of  Setup \ref{setup}.
The category $\Q$ has only one object, which we denote by $``*"$. In this case, the functor $E_{*}$ and $F_{*}$ from \ref{adjoint triple}
  are given by $$E_{*}(X,d_{X})=X \quad\text{and}\quad F_*(M)=(M\oplus M,\left( \begin{smallmatrix} 0 & 0 \\1 & 0 \\\end{smallmatrix} \right)),$$
 and the functors $C_{*}$ and $K_{*}$ from \ref{adjoint triple2}
act on a differential module $(X,d_{X})$ by taking cokernel and kernel of its differential, respectively, i.e. $$C_{*}(X,d_{X})=\Coker d_{X} \quad \text{and}\quad K_{*}(X,d_{X})=\Ker d_{X}.$$
\end{rmk}

\begin{bfhpg}[\bf Homology]
For every differential left $A$-module $(X,d_{X})$, since $d_{X}d_{X}=0$, we have $\mathrm{Im} d_{X}\subseteq \mathrm{Ker} d_{X}.$
The quotient module
$$\mathrm{H}(X,d_{X})=\Ker d_{X}/\mathrm{Im} d_{X}$$ is the \emph{homology} of $(X,d_{X})$.  By \cite[A.2]{HJ23}, $(X,d_{X})$ is
 exact in the sense of \ref{exact object} if and only if $\mathrm{H}(X,d_{X})=0.$ This is a special case of exact objects. In what follows, we set $$\mathscr{E}=\{(X,d_{X})\in \Di\mid  (X,d_{X})~\text{is exact}\}.$$
 \end{bfhpg}

 Recall from \cite{LLAHBF2007} that the \emph{suspension functor} $\Sigma$ on $\Di$ is defined in the following way: $\Sigma(Y,d_{Y})=(Y,-d_{Y})$ and $\Sigma f=f$.
 It is easy to check that there exists an exact sequence in $\Di$
\begin{equation}\label{se.suspension}
 0\to (Y,d_{Y})\xrightarrow{\left( \begin{smallmatrix} d_{Y} \\1 \\\end{smallmatrix} \right)} (Y\oplus Y,\left( \begin{smallmatrix} 0 & 0 \\1 & 0 \\\end{smallmatrix} \right))\xrightarrow{(1,-d_{Y})} (Y,-d_{Y})\to 0.
\end{equation}

\begin{dfn}
A morphism $f:(X,d_{X})\to (Y,d_{Y})$ between two differential modules is null-homotopical if there exists a morphism $h:X\to Y$ such that $f=d_{Y}h+hd_{X}.$
\end{dfn}

Recall from that \cite{DN} that an object $(X,d_{X})$ in $\Di$ is \emph{contractible} if there exists an endomorphism $X\xrightarrow{h}X$ in $\Mod{A}$ such that the equation $\mathrm{id}_{X}=hd_{X}+d_{X}h$ holds.
 We call $h$ a contractible homotopy of $(X,d_{X}).$
 By \cite[1.5 and Proposition 1.8]{LLAHBF2007} or \cite[Proposition 3,6]{DN}, a differential module $(X,d_{X})$ is contractible if and only if it is isomorphic to $F_{*}(M)=(M\oplus M,\left( \begin{smallmatrix} 0 & 0 \\1 & 0 \\\end{smallmatrix} \right))$ for some $M\in\Mod{A}.$

\begin{lem}\label{lem:contractible object}
A morphism $f:(X,d_{X})\to (Y,d_{Y})$ in $\Di$ is null-homotopical if and only if it factors through a contractible object $(M\oplus M,\left( \begin{smallmatrix} 0 & 0 \\1 & 0 \\\end{smallmatrix} \right))$ for some $M\in\Mod{A}.$
\end{lem}
\begin{proof}
Assume that $f$ factors through a contractible object $(M\oplus M,\left( \begin{smallmatrix} 0 & 0 \\1 & 0 \\\end{smallmatrix} \right))$ for some $M\in \Mod{A}.$
Then we get the following commutative diagram:
$$
\xymatrix{X\ar[r]^{d_{X}}\ar[d]_{\left( \begin{smallmatrix} f_{0} \\f_{1} \\\end{smallmatrix} \right)}\ar@/_3.6pc/[dd]_{f}&X\ar[d]\ar@/^3.6pc/[dd]^{f}\ar[d]^{\left( \begin{smallmatrix} f_{0} \\f_{1} \\\end{smallmatrix} \right)}\\
M\oplus M\ar[r]^{\left( \begin{smallmatrix} 0 & 0 \\1 & 0 \\\end{smallmatrix} \right)}\ar[d]_{(g_{0},g_{1})}&M\oplus M\ar[d]^{(g_{0},g_{1})}\\
Y\ar[r]^{d_{Y}}&Y
}$$
It implies that $f_{0}=f_{1}d_{X}$ and $g_{1}=d_{Y}g_{0}.$
 Hence
 \begin{align*}
 f=(g_{0},g_{1})\left( \begin{smallmatrix} f_{0} \\f_{1} \\\end{smallmatrix} \right)=g_{0}f_{0}+g_{1}f_{1}=(g_{0}f_{1})d_{X}+d_{Y}(g_{0}f_{1}),
\end{align*}
which implies that $f$ is null-homotopical.

Conversely, assume that there exists a morphism $h:X\to Y$ in $\Mod{A}$ such that $f=d_{Y}h+hd_{X}$.
Then we obtain the following commutative diagram
$$
\xymatrix{X\ar[r]^{d_{X}}\ar[d]_{\left( \begin{smallmatrix} hd_{X} \\h \\\end{smallmatrix} \right)}&X\ar[d]\ar[d]^{\left( \begin{smallmatrix} hd_{X} \\h \\\end{smallmatrix} \right)}\\
Y\oplus Y\ar[r]^{\left( \begin{smallmatrix} 0 & 0 \\1 & 0 \\\end{smallmatrix} \right)}\ar[d]_{(1,d_{Y})}&Y\oplus Y\ar[d]^{(1,d_{Y})}\\
Y\ar[r]^{d_{Y}}&Y.
}$$
It implies that $f$ factors through the contractible object $(Y\oplus Y,\left( \begin{smallmatrix} 0 & 0 \\1 & 0 \\\end{smallmatrix} \right).$
Consequently, we complete the proof of the lemma.
\end{proof}
In what follows, by Lemma \ref{lem:contractible object}, we denote by $\underline{\mathsf{Dif}}(A)$ the stable category modulo null-homotopical morphisms or those morphisms factoring through contractible objects.
\end{bfhpg}

\begin{bfhpg}[\bf Frobenius exact structure on $\Di$]
Let $\mathcal{S}$ be the class of exact sequences
$$0\to (X,d_{X})\to (Y,d_{Y})\to (Z,d_{Z})\to 0$$
in $\Di$ such that the exact sequence
$$0\to X\to Y\to Z\to 0$$
in $\Mod{A}$ splits.
Then it follows from \cite[Remark 4.8]{SV} that $(\Di, \mathcal{S})$ is a Frobenius exact category, and its exact structure is called the objectwise split exact structure.
In the following, for any $(X,d_{X})$ and $(Y,d_{Y})$ in $\Di$, we denote by $\Ext[dw]{1}{(X,d_{X})}{(Y,d_{Y})}$ the Yoneda Ext group associated to the Frobenius exact category $(\Di, \mathcal{S})$.

\begin{rmk}\label{rm:homotopy category}
Consider the Frobenius exact category $(\Di,\mathcal{S})$. Then there exists a Frobenius exact model structure $(\Di,\mathsf{Inj}(\Di,\mathcal{S}),\Di)$ in the sense of \cite[Definition 4.5(3)]{Gi11}, where $\mathsf{Inj}(\Di,\mathcal{S})$ is the subcategory of injective objects in $(\Di,\mathcal{S})$.
By the proof of \cite[Proposition 4.9]{SV}, $\mathsf{Inj}(\Di,\mathcal{S})$ coincides with the class of contractible objects.
It follows from \cite[Corollary 4.8(3)]{Gi11} and Lemma \ref{lem:contractible object} that the homotopy category of the exact model structure coincides with $\underline{\mathsf{Dif}}(A)$.
\end{rmk}

\end{bfhpg}

\begin{lem}\label{homotopic}
Let $(X,d_{X})$ and $(Y,d_{Y})$ be two differential left $A$-modules. Then we have the following isomorphism
$$ \Ext[dw]{1}{(X,d_{X})}{(Y,d_{Y})}\cong \mathrm{Hom}_{\underline{\sf Dif}(A)}((X,d_{X}),(Y,-d_{Y})).$$
\end{lem}
\begin{proof}
Consider the exact sequence (\ref{se.suspension}) in $\Di$
 $$0\to (Y,d_{Y})\xrightarrow{\left( \begin{smallmatrix} d_{Y} \\1 \\\end{smallmatrix} \right)} (Y\oplus Y,\left( \begin{smallmatrix} 0 & 0 \\1 & 0 \\\end{smallmatrix} \right))\xrightarrow{(1,-d_{Y})} (Y,-d_{Y})\to 0.$$
Note that $(Y\oplus Y,\left( \begin{smallmatrix} 0 & 0 \\1 & 0 \\\end{smallmatrix} \right))$ is a contractible object, and so it is an injective object in $(\Di, \mathcal{S})$ by Remark \ref{rm:homotopy category}.
Then this result follows from Remark \ref{rm:homotopy category} and \cite[Theorem 8.18]{Gil24}.
\end{proof}

\begin{nota}\label{nota}
Let $(\mathcal{X}, \mathcal{Y})$ be a cotorsion pair in $\Mod{A}$. The subcategories defined in Definition \ref{df:cotorsion pair} take the following form in $\Di$ by Proposition \ref{prop:KC} and Remark \ref{rm:Jordan}:
$$\quad \ \ \widetilde{\mathcal{X}}=\{(X,d_{X})\in \Di\ |\ (X,d_{X})~\text{is exact and}\ \Coker d_{X}\in \mathcal{X}\},$$
$$\quad \ \ \widetilde{\mathcal{Y}}=\{(Y,d_{Y})\in \Di\ |\ (Y,d_{Y})~\text{is exact and}\ \Ker d_{Y}\in \mathcal{Y}\}.$$
Keep the same notation as in Definition \ref{df:cotorsion pair}, we still let
$\mathsf{dg} \mathcal{X}={^{\bot}\widetilde{\mathcal{Y}}}$ and $\mathsf{dg} \mathcal{Y}=\widetilde{\mathcal{X}}^{\bot}$.
\end{nota}

\begin{rmk}\label{rm:suspension}
Following Notaition \ref{nota}, it is straightforward to check that $\WX$ and $\WY$ are closed under suspensions.
Hence $\DX$ and $\DY$ are closed under suspensions. Indeed, take $(X,d_{X})\in \DX$, it follows from Corollary \ref{cor:dw} that $X\in \X$.
For any $(Y,d_{Y})\in \WY$, one has $Y\in \Y$. Thus there exists the following isomorphisms
\begin{align*}
\Ext[\Q,A]{1}{\Sigma(X,d_{X})}{(Y,d_{Y})}&\cong \Ext[dw]{1}{(X,-d_{X})}{(Y,d_{Y})}\\
&\cong\mathrm{Hom}_{\underline{\sf Dif}(A)}((X,-d_{X}),(Y,-d_{Y}))\\
&\cong\mathrm{Hom}_{\underline{\sf Dif}(A)}((X,d_{X}),(Y,d_{Y}))\\
&\cong\Ext[dw]{1}{(X,d_{X})}{(Y,-d_{Y})}\\
&\cong\Ext[\Q,A]{1}{(X,d_{X})}{(Y,-d_{Y})}\\
&\cong 0,
\end{align*}
where the first isomorphism follows from the definition of $\Sigma$ and the fact that $X\in\X$ and $Y\in \Y,$  the second and the forth isomorphisms hold by Lemma \ref{homotopic}.
Consequently, $\Sigma(X,d_{X})=(X,-d_{X})\in \DX.$
\end{rmk}

\begin{lem}\label{lem:orthogonal}
Let $(\mathcal{X}, \mathcal{Y})$ be a cotorsion pair in $\Mod{A}$.
Then $$\Ext[\Di]{1}{(X,d_{X})}{(Y,d_{Y})}=0$$ for any $(X,d_{X})\in \widetilde{\mathcal{X}}$ and $(Y,d_{Y})\in \widetilde{\mathcal{Y}}$.
\end{lem}
\begin{proof}
 Suppose that $(X,d_{X})\in \widetilde{\mathcal{X}}$ and $(Y,d_{Y})\in \widetilde{\mathcal{Y}}$. Then it follows from Notation \ref{nota} that
there exist two exact sequences in $\Mod{A}$
$$0\to \Coker d_{X}\to X\to \Coker d_{X}\to 0$$
and
$$0\to \Ker d_{Y}\to Y\to \Ker d_{Y}\to 0$$
with $\Coker d_{X}\in \mathcal{X}$ and $\Ker d_{Y}\in \mathcal{Y}.$
Note that $\mathcal{X}$ and $\mathcal{Y}$ are closed under extensions as $(\mathcal{X},\mathcal{Y})$ is a cotorsion pair.
Hence $X\in \mathcal{X}$ and $Y\in \mathcal{Y}.$
It implies that $$\Ext[\Di]{1}{(X,d_{X})}{(Y,d_{Y})}=\Ext[dw]{1}{(X,d_{X})}{(Y,d_{Y})}.$$
By Lemma \ref{homotopic}, it suffices to prove that $$\Hom[\underline{\sf Dif}(A)]{(X,d_{X})}{(Y,d_{Y})}=0$$ for any $(X,d_{X})\in\WX$ and $(Y,d_{Y})\in\WY.$
Let $f:(X,d_{X})\to (Y,d_{Y})$ be a morphism in $\Di.$
Then we have the following commutative diagram
$$
\xymatrix{X\ar[rr]^{d_{X}}\ar[rd]^{\widetilde{d_{X}}}\ar[dd]_{f}& &X\ar[dd]^{f}\\
&\Ker d_{X}\ar[ur]^{i_{X}}\ar[dd]^(.3){\widehat{f}}&\\
Y\ar@{.>}[rr]^(.4){d_{Y}}\ar[rd]_{\widetilde{d_{Y}}}& & Y\\
&\Ker d_{Y}\ar[ur]_{i_{Y}}&
}$$
where $\widehat{f}:\Ker d_{X}\to \Ker d_{Y}$ is the restriction of $f:X\to Y.$
Consider the exact sequence
$$0\to \Ker d_{Y}\xrightarrow{i_{Y}} Y\xrightarrow{\widetilde{d_{Y}}} \Ker d_{Y}\to 0.$$
Since $(X,d_{X})\in \WX$ and $(Y,d_{Y})\in \WY$, by Notation \ref{nota}, one has $\Ker d_{X}=\Coker d_{X}\in \X$ and $\Ker{d_{Y}}\in \Y$.
It follows that $$\Ext[A]{1}{\Ker{d_{X}}}{\Ker{d_{Y}}}=0.$$
Thus we get the exact sequence
$$0\to \Hom[A]{\Ker d_{X}}{\Ker d_{Y}}\to \Hom[A]{\Ker d_{X}}{Y}\to \Hom[A]{\Ker d_{X}}{\Ker d_{Y}}\to 0.$$
It follows that there exists a morphism $\alpha\in \Hom[A]{\Ker{d_{X}}}{Y}$ with $\widehat{f}=\widetilde{d_{Y}}\alpha.$
Note that $$0\to \Ker d_{X}\xrightarrow{i_{X}} X\xrightarrow{\widetilde{d_{X}}} \Ker d_{X}\to 0$$ is an exact sequence as $(X,d_{X})$ is exact.
Since $\Ker{d_{X}}\in \mathcal{X}$ and $Y\in \Y$, one has $\Ext[A]{1}{\Ker{d_{X}}}{Y}=0$.
Then we get the exact sequence
$$0\to \Hom[A]{\Ker d_{X}}{Y}\xrightarrow{\widetilde{d_{X}}^{*}}\Hom[A]{X}{Y}\xrightarrow{i_{X}^{*}}\Hom[A]{\Ker d_{X}}{Y}\to 0.$$
So there exists a morphism $\beta\in \Hom[A]{X}{Y}$ with $\beta i_{X}=\alpha.$
Now set $g=f-(d_{Y}\beta+\beta d_{X})$.
We show that $g$ is a morphism between $(X,d_{X})$ and $(Y,d_{Y})$.
Indeed,
\begin{align*}
gd_{X}&=[f-(d_{Y}\beta+\beta d_{X})]d_{X}\\
&=fd_{X}-d_{Y}\beta d_{X}\\
&=fd_{X}-d_{Y}\beta i_{X}\widetilde{d_{X}}\\
&=fd_{X}-i_{Y}\widetilde{d_{Y}}\alpha\widetilde{d_{X}}\\
&=fd_{X}-i_{Y}\widehat{f}\widetilde{d_{X}}\\
&=fd_{X}-fi_{X}\widetilde{d_{X}}\\
&=0.
\end{align*}
By a similar argument,  $d_{Y}g=0=gd_{X}$.
Hence $g$ is a morphism in $\Di$, and $f-g=d_{Y}\beta+\beta d_{X}.$
We mention that $\mathrm{Im} g\subseteq\Ker {d_{Y}}$ and $\mathrm{Im} d_{X}\subseteq \Ker {g}.$
Thus we have the following commutative diagram
$$
\xymatrix{X\ar[d]^{\widetilde{d_{X}}}\ar@{=}[r]&X\ar[d]^{\pi}\ar@{=}[r]&X\ar[d]^{g}\\
\Ker d_{X}&X/\Ker d_{X}\ar[l]^{\overline{d_{X}}}_{\cong}\ar[r]^{\overline{g}}&\Ker d_{Y}.
}
$$
Set $\widehat{g}=\overline{g}\overline{d_{X}}^{-1},$ and then
$\widehat{g}\widetilde{d_{X}}=g.$
Now using the argument as previously to obtain the morphism $\alpha$, there exists a morphism $\delta:\Ker d_{X}\to Y$ with $\widetilde{d_{Y}}\delta=\widehat{g}.$
Hence $$d_{Y}(\delta d_{X})+(\delta d_{X})d_{X}=d_{Y}\delta d_{X}=i_{Y}\widetilde{d_{Y}}\delta d_{X}=i_{Y}\widehat{g}d_{X}=g.$$
It follows that
\begin{align*}
f&=d_{Y}\beta+\beta d_{X}+d_{Y}(\delta d_{X})+(\delta d_{X})d_{X}\\
&=d_{Y}(\beta+\delta d_{X})+(\beta+\delta d_{X})d_{X}.
\end{align*}
Consequently, $f$ is null-homotopical.
\end{proof}

Recall from \cite{DN} that an object $(I,d_{I})$ in $\Di$ is semi-injective if for each exact object $(E,d_{E})$ in $\Di$, we have $\Ext[\Di]{1}{(E,d_{E})}{(I,d_{I})}=0.$ This is a special case of the semi-injecive objects defined in \ref{exact object}.

\begin{prp}\label{prop:complete cotorsion pair}
Assume that $(\mathcal{X},\mathcal{Y})$ is a complete cotorsion pair in $\Mod{A}.$
Then  $(\widetilde{\mathcal{X}},\mathsf{dg}\mathcal{Y})$ and $(\mathsf{dg}\mathcal{X},\widetilde{\mathcal{Y}})$ are complete  cotorsion pairs in $\Di.$
\end{prp}
\begin{proof}
We already know that $(\widetilde{\mathcal{X}},\mathsf{dg}\mathcal{Y})$ and $(\mathsf{dg}\mathcal{X},\widetilde{\mathcal{Y}})$ are  cotorsion pairs in $\Di$ by Proposition \ref{prop:cotorsion pair}.
We need only to show the cotorsion pair $(\widetilde{\mathcal{X}},\mathsf{dg}\mathcal{Y})$ is complete,
because the
completeness of the cotorsion pair $(\mathsf{dg}\mathcal{X},\widetilde{\mathcal{Y}})$ can be proved dually.
Take $(M,d_{M})\in \Di$.
By \cite[Theorem 5.9]{HJ21}, there exists an exact sequence
$$0\to (J,d_{J})\to (E,d_{E})\to (M,d_{M})\to 0$$
with $(E,d_{E})$ exact and $(J,d_{J})$ semi-injective in $\Di.$
Since $(\mathcal{X},\mathcal{Y})$ is complete, we have an exact sequence
$$0\to Y\to X\to \Ker d_{E}\to 0$$
with $X\in \mathcal{X}$ and $Y\in\mathcal{Y}.$
We mention that $$0\to \Ker d_{E}\to E\to \Ker d_{E}\to 0$$
is exact as $(E,d_{E})$ is exact.
By Generalized Horseshoe Lemma (see \cite[Lemma 2.18]{Gil24}), we have the following exact commutative diagram
$$\xymatrix{&  0\ar[d]& 0\ar[d]&0\ar[d]&\\
 0\ar[r]&Y\ar[r]^{i_{Y'}}\ar[d]_{}&Y'\ar[r]^{\widetilde{d_{Y'}}}\ar[d]_{}&Y\ar[r]\ar[d]_{}&0\\
 0\ar[r]&X\ar[r]^{i_{X'}}\ar[d]&X'\ar[r]^{\widetilde{d_{X'}}}\ar[d]&X\ar[r]\ar[d]&0\\
 0\ar[r]&\Ker d_{E}\ar[r]^{i_{E}}\ar[d]&E\ar[r]^{\widetilde{d_{E}}}\ar[d]&\Ker d_{E}\ar[r]\ar[d]&0\\
 &0&0&0&}$$
 with $X'\in \mathcal{X}$ and $Y'\in \mathcal{Y}.$
 We set $d_{X'}=i_{X'}\widetilde{d_{X'}}$ and $d_{Y'}=i_{Y'}\widetilde{d_{Y'}},$
 and then $(X',d_{X'})$ and $(Y',d_{Y'})$ are in $\Di.$
 The above commutative diagram implies that there is an exact sequence in $\Di$
 $$0\to (Y',d_{Y'})\to (X',d_{X'})\to (E,d_{E})\to 0$$
 with $(X',d_{X'})\in \widetilde{\mathcal{X}}$ and $(Y',d_{Y'})\in \widetilde{\mathcal{Y}}.$
 Consider the following pullback diagram in $\Di:$
 $$\xymatrix{&0\ar[d]&0\ar[d]&\\
&(Y',d_{Y'})\ar[d]\ar@{=}[r]&(Y',d_{Y'})\ar[d]&\\
0\ar[r]&(K,d_{K})\ar[r]\ar[d]&(X',d_{X'})\ar[r]\ar[d]&(M,d_{M})\ar[r]\ar@{=}[d]&0\\
    0\ar[r]&(J,d_{J})\ar[r]\ar[d]&(E,d_{E})\ar[d]\ar[r]&(M,d_{M}) \ar[r]&0\\
    &0&0&\\}$$
Since every object in $\widetilde{\mathcal{X}}$ is exact, it follows  that $(J,d_{J})$ is in $\mathsf{dg}\mathcal{Y}$.
Since $\widetilde{\mathcal{Y}}\subseteq\mathsf{dg}\mathcal{Y}$ by Lemma \ref{lem:orthogonal}, one has $(Y',d_{Y'})\in \mathsf{dg}\mathcal{Y}.$
Note that $\mathsf{dg}\mathcal{Y}$ is closed under extensions. It follows from the first column of above pullback diagram that $(K,d_{K})\in \mathsf{dg}\mathcal{Y}$, proving that the cotorsion pair $(\widetilde{\mathcal{X}},\mathsf{dg}\mathcal{Y})$ has enough projectives.
Since $\Di$ has enough projectives and injectives, it follows from
\cite[Proposition 7.1.7]{rha} that $(\widetilde{\mathcal{X}},\mathsf{dg}\mathcal{Y})$ is complete.
\end{proof}

\begin{prp}\label{prop:com}
Let $(\mathcal X,\mathcal Y)$ be a hereditary cotorsion pair in $\Mod{A}$.
Then one has $$
\mathsf{dg}\mathcal{X}\cap \mathscr{E}
=\widetilde{\mathcal X}.$$
\end{prp}

\begin{proof}
By Lemma \ref{lem:orthogonal}, we have $\widetilde{\X}\subseteq \mathsf{dg}\mathcal{X}\cap \mathscr{E}.$
It is enough to prove the inclusion
\[
\mathsf{dg}\mathcal{X}\cap \mathscr{E}
\subseteq
\widetilde{\mathcal X}.\]
Let$(M,d_{M})\in \mathsf{dg}\mathcal{X}\cap\mathscr{E}.$
Since $(M,d_{M})$ is exact, there is an exact sequence of left $A$-modules
\begin{equation*}\label{comp}
\tag{\ref{prop:com}.1}
0\to \Ker d_{M}\xrightarrow{i_{M}} M\xrightarrow{\widetilde{d_{M}}} \Ker d_{M}\to 0.
\end{equation*}
We need to show that $\Ker d_{M}\in\mathcal X$.
Since $(\mathcal X,\mathcal Y)$ is a cotorsion pair, it suffices to prove that
\[
\operatorname{Ext}_A^1(\Ker d_{M},Y)=0
\]
for every $Y\in\mathcal Y$.

Now Fix $Y\in\mathcal Y$.
By Corollary \ref{cor:dw}, the left $A$-module $M$ belongs to $\mathcal X$. Hence
\[
\operatorname{Ext}_A^1(M,Y)=0.\]
Applying $\operatorname{Hom}_A(-,Y)$ to (\ref{comp}), we obtain the following exact sequence
\[
\operatorname{Hom}_A(M,Y)
\xrightarrow{i_{M}^{*}}
\operatorname{Hom}_A(\Ker d_{M},Y)
\longrightarrow
\operatorname{Ext}_A^1(\Ker d_{M},Y)
\longrightarrow
\operatorname{Ext}_A^1(M,Y)=0.
\]
Therefore, it remains to prove that every morphism
$f:\Ker d_{M}\longrightarrow Y$  extends to a morphism $M\to Y$.

Choose an injective resolution of $Y$,
\begin{equation}\label{comp.2}
\tag{\ref{prop:com}.2}
0\longrightarrow Y
\xrightarrow{\eta}
I^0
\xrightarrow{\delta^0}
I^1
\xrightarrow{\delta^1}
I^2
\xrightarrow{\delta^2}
\cdots .
\end{equation}
For convenience, put
\[
B^{-1}=Y,
\qquad
B^n=I^n \quad (n\geq 0),
\qquad
\delta^{-1}=\eta.
\]
Thus (\ref{comp.2}) can be written as the exact complex
\[
0\longrightarrow
B^{-1}
\xrightarrow{\delta^{-1}}
B^0
\xrightarrow{\delta^0}
B^1
\xrightarrow{\delta^1}
B^2
\longrightarrow\cdots .
\]
Since $(\mathcal X,\mathcal Y)$ is hereditary, the class $\mathcal Y$ is
coresolving. Consequently, all cosyzygies $\Ker\delta^n$ belong to $\mathcal Y$.

We now fold this injective resolution into a differential module.
Set
\[
T(Y)
=
\prod_{n\geq -1} B^n
=
Y\times\prod_{n\geq 0} I^n.
\]
Define an endomorphism
\[
\prod_{n\geq -1}\delta^{n}:T(Y)\longrightarrow T(Y)
\]
by
\[
\prod_{n\geq -1}\delta^{n}(b_{-1},b_0,b_1,b_2,\ldots)
=
\bigl(
0,
\delta^{-1}(b_{-1}),
\delta^0(b_0),
\delta^1(b_1),
\ldots
\bigr).
\]
Since $\delta^n\delta^{n-1}=0$ for every $n\geq 0$, one has $(\prod_{n\geq -1}\delta^{n})^2=0$. Hence $(T(Y),\prod_{n\geq -1}\delta^{n})$ is a differential left $A$-module. Moreover, by \cite[Corollary 3.8]{DN}, $(T(Y),\prod_{n\geq -1}\delta^{n})$ is exact.

Furthermore, $\Ker (\prod_{n\geq -1}\delta^{n}) \cong \prod_{n\geq -1}\Ker\delta^n.$
Since each $\Ker\delta^n$ belongs to $\mathcal Y$ and $\Y$ is closed under products, we obtain $\Ker (\prod_{n\geq -1}\delta^{n})\in\mathcal Y.$
It follows that $ (T(Y),\prod_{n\geq -1}\delta^{n})\in\widetilde{\mathcal Y}.$
Now let $ f:\Ker d_{M}\longrightarrow Y$
be arbitrary in $\Di$. We shall construct a morphism of differential left $A$-modules
\[
F:(M,d_{M})\longrightarrow (T(Y),\prod_{n\geq -1}\delta^{n}).
\]

Define first
\[
F_{-1}=f\widetilde{d_{M}}:M\longrightarrow Y.
\]
Since $d_{M}^2=0$, one has $F_{-1}d_{M}=0.$
Next, the map $
\eta f:Z\longrightarrow I^0$
extends, by injectivity of $I^0$, to a morphism $F_0:M\longrightarrow I^0$
such that $F_0|_{\Ker d_{M}}=\eta f.$
For every $m\in M$, we have $\widetilde{d_{M}}(m)\in \Ker d_{M}$, and hence
\[
F_0d_{M}(m)=F_0\widetilde{d_{M}}(m)
=
\eta f \widetilde{d_{M}}(m)
=
\eta F_{-1}(m).
\]
Thus $F_0d_{M}=\delta^{-1}F_{-1}.$

Suppose inductively that, for some $n\geq 0$, morphisms
\[
F_j:M\longrightarrow I^j
\qquad (0\leq j\leq n)
\]
have been constructed and satisfy
\[
F_jd=\delta^{j-1}F_{j-1}
\qquad (1\leq j\leq n).
\]
We construct $F_{n+1}:M\to I^{n+1}$. Consider the morphism
\[
\delta^nF_n:M\longrightarrow I^{n+1}.
\]
We claim that it vanishes on $\Ker d_{M}$. Indeed, if $m\in \Ker d_{M}$, then,
since $(M,d_{M})$ is exact, there exists $m'\in M$ such that $m=d_{M}(m')$. Therefore
\[
\delta^nF_n(m)
=
\delta^nF_nd_{M}(m')
=
\delta^n\delta^{n-1}F_{n-1}(m')
=
0.
\]
Hence $\delta^nF_n$ factors uniquely through the epimorphism
$\widetilde{d_{M}}:M\to \Ker d_{M}$, yielding a morphism
$\overline F_{n+1}:\Ker d_{M}\longrightarrow I^{n+1}$ in $\Mod{A}$
such that $\overline F_{n+1}\widetilde{d_{M}}=\delta^nF_n.$
Since $I^{n+1}$ is injective, $\overline F_{n+1}$ extends to a morphism
$F_{n+1}:M\longrightarrow I^{n+1}.$
Then $F_{n+1}d_{M}=\delta^nF_n.$
Proceeding inductively, we obtain a homomorphism
\[
F=\left( \begin{matrix} F_{-1} \\F_0\\F_1\\\vdots \\\end{matrix} \right)
:
M\longrightarrow T(Y)
\]
satisfying
\[
(\prod_{n\geq -1}\delta^{n})F=Fd_{M}.
\]
Thus $F:(M,d_{M})\longrightarrow (T(Y),\prod_{n\geq -1}\delta^{n})$ is a morphism in $\Di$.

Since $M\in \mathsf{dg}\mathcal{X}$ and
$T(Y)\in\widetilde{\mathcal Y},$ it follows from Lemmas \ref{homotopic} that
the morphism $F$ is null-homotopic. Hence there exists a morphism
$H:M\longrightarrow T(Y)$ in $\Mod{A}$
such that $$F=(\prod_{n\geq -1}\delta^{n})H+Hd_{M}.$$
Write
\[
H=\left( \begin{matrix} H_{-1} \\H_0\\H_1\\\vdots \\\end{matrix} \right),
\]
where $H_{-1}:M\longrightarrow Y$ is a morphism in $\Mod{A}.$
Since the $Y$-component of $\prod_{n\geq -1}\delta^{n}$ is zero, taking the $(-1)$-st component of
the equality $F=(\prod_{n\geq -1}\delta^{n})H+Hd_{M}$ gives
$F_{-1}=H_{-1}d_{M}.$
On the other hand, $F_{-1}=f\widetilde{d_{M}}.$
Hence $f\widetilde{d_{M}}=H_{-1}d_{M}.$
Since $(M,d_{M})$ is exact, it follows that $H_{-1}d_{M}=H_{-1}\widetilde{d_{M}}$.
So $f\widetilde{d_{M}}=H_{-1}\widetilde{d_{M}}$. Note that $\widetilde{d_{M}}:M\to \Ker d_{M}$ is an epimorphism in $\Mod{A}.$ Then $H_{-1}:M\to Y$ is an extension of $f:\Ker d_{M}\to Y$.

We have therefore proved that $i_{M}^{*}:\operatorname{Hom}_A(M,Y)
\to
\operatorname{Hom}_A(\Ker d_{M},Y)$
is surjective. Consequently, $\operatorname{Ext}_A^1(\Ker d_{M},Y)=0.$
Since $Y\in\mathcal Y$ was arbitrary, we conclude that $\Ker d_{M}\in{}^\perp\mathcal Y=\mathcal X.$
As $(M,d_{M})$ is exact, this shows that $(M,d_{M})\in\widetilde{\mathcal X}.$
Therefore $\mathsf{dg}\mathcal{X}\cap\mathscr{E}
\subseteq
\widetilde{\mathcal X},$
and hence $
\mathsf{dg}\mathcal{X}\cap\mathscr{E}
=
\widetilde{\mathcal X}.$
\end{proof}

\begin{lem}\label{lem:compatible}
If $(\mathcal{X},\mathcal{Y})$ is a  cotorsion pair on $\Mod{A},$ then $(\mathsf{dg}\mathcal{X},\widetilde{\mathcal{Y}})$ and $(\widetilde{\mathcal{X}},\mathsf{dg}\mathcal{Y})$ are compatible cotorsion pairs. In particular,
$$\mathsf{dg}\mathcal{X}\cap\widetilde{\mathcal{Y}}= \Big\{(X,d_{X})\in \Di\ |\ (X,d_{X})\ \text{is contractible and}\ X\in \mathcal{X}\cap\mathcal{Y}\Big\}=\widetilde{\mathcal{Y}}\cap\mathsf{dg}\mathcal{X}.$$
\end{lem}
\begin{proof}
By Proposition \ref{prop:cotorsion pair}, we know that $(\mathsf{dg}\mathcal{X},\widetilde{\mathcal{Y}})$ and $(\WX,\mathsf{dg}\mathcal{Y})$ are cotorsion pairs in $\Di$.
Next we show that
$$\mathsf{dg}\mathcal{X}\cap\widetilde{\mathcal{Y}}= \Big\{(X,d_{X})\in \Di\ |\ (X,d_{X})\ \text{is contractible and}\ X\in \mathcal{X}\cap\mathcal{Y}\Big\}.$$
Take a contractible object $(X,d_{X})$ with $X\in \mathcal{X}\cap\mathcal{Y}.$
Then $$(X,d_{X})\cong F_{*}(M)=(M\oplus M,\left( \begin{smallmatrix} 0 & 0 \\1 & 0 \\\end{smallmatrix} \right))$$ for some $M\in\Mod{A}$ by \cite[1.5 and Proposition 1.8]{LLAHBF2007} or \cite[Proposition 3.6]{DN}.
Hence $M\in \mathcal{X}\cap\mathcal{Y}.$
It follows from Lemma \ref{lem:contain contractible objects} that $(X,d_{X})\in \mathsf{dg}\mathcal{X}\cap\widetilde{\mathcal{Y}}.$
On the other hand, suppose that $(M,d_{M})\in \mathsf{dg}\mathcal{X}\cap\widetilde{\mathcal{Y}}$.
Then, by Remark \ref{rm:suspension}, $\Sigma(M,d_{M})=(M,-d_{M})\in\mathsf{dg}\mathcal{X}\cap\widetilde{\mathcal{Y}}.$
Since $(\mathsf{dg}\mathcal{X},\widetilde{\mathcal{Y}})$ is a cotorsion pair, one has $\Ext[\Di]{1}{(M,d_{M})}{(M,-d_{M})}=0$, and so Lemma \ref{homotopic} gives us
$$ \Ext[dw]{1}{(M,d_{M})}{(M,-d_{M})}\cong \mathrm{Hom}_{\underline{\sf Dif}(A)}((M,d_{M}),(M,d_{M}))=0.$$
It follows that $(M,d_{M})$ is a contractible object.
By \cite[Proposition 1.8]{LLAHBF2007} or \cite[Proposition 3.6]{DN},
we have $M\in \mathcal{X}\cap \mathcal{Y}.$

By a dual argument as above, we can show that $$\widetilde{\mathcal{Y}}\cap \mathsf{dg}\mathcal{X}= \Big\{(X,d_{X})\in \Di\ |\ (X,d_{X})\ \text{is contractible and}\ X\in \mathcal{X}\cap\mathcal{Y}\Big\}.$$
Thus we complete the proof.
\end{proof}

\begin{thm} \label{thm:main result in D}
Let $(\mathcal{X},\mathcal{Y})$ be a complete hereditary cotorsion pair in $\Mod{A}$.
Then there is an abelian model structure on $\Di$ whose cofibrant, fibrant, and trivial objects are $\mathsf{dg}\mathcal{X}$, $\mathsf{dg}\mathcal{Y}$, and
$\mathscr{E}$ , respectively. Consequently, its homotopy category is equivalent to the $\Q$-shaped derived category, where $\Q$ is the path category of the Jordan quiver
\vspace{-2mm}
\[
\begin{tikzpicture}[baseline=-0.5ex]
\node (v) at (0,0) {$\bullet$};
\draw[->]
(v) to[out=35,in=-35,looseness=7]
node[right] {$\varepsilon$} (v);
\end{tikzpicture}
\]
\vspace{-2mm}
subject to the relation $\varepsilon^{2}=0$.
\end{thm}
\begin{proof}
By Corollary \ref{cor:hereditary cotorsion pair}, Proposition \ref{prop:complete cotorsion pair} and Lemma \ref{lem:compatible},
$(\widetilde{\mathcal{X}},\mathsf{dg}\mathcal{Y})$ and $(\mathsf{dg}\mathcal{X},\widetilde{\mathcal{Y}})$ are complete hereditary, and compatible cotorsion pairs in $\Di$.
By \cite[Theorem 1.1]{G15} or Remark \ref{rmk:projective model structure}, it remains to verify that the associated class $\mathcal W$ of trivial objects coincides with $\mathscr{E}.$
Here
\begin{align*}
\mathcal{W} = \Big\{(M,d_M)\in \Di \mid {}&
\exists\ s.e.s\
 0\to(Y,d_Y)\to (X,d_X)\to (M,d_M)\to 0\\
 &~~\text{with } (Y,d_Y)\in \widetilde{\mathcal{Y}} ~~\text{and}~~(X,d_X)\in \widetilde{\mathcal{X}} \Big\},\\
 =\Big\{(M,d_M)\in \Di \mid {}&
\exists\ s.e.s\
0\rightarrow (M,d_M)\rightarrow (Y',d_{Y'})\rightarrow (X',d_{X'})\rightarrow 0 \\
 &~~\text{with } (Y',d_{Y'})\in \widetilde{\mathcal{Y}} ~~\text{and}~~(X',d_{X'})\in \widetilde{\mathcal{X}} \Big\}.
\end{align*}
Since $\widetilde{\mathcal X},\widetilde{\mathcal Y}\subseteq\mathscr{E}$
and $\mathscr{E}$ is thick, either of the above short exact sequences
shows immediately that $
\mathcal W\subseteq\mathscr{E}.$

Conversely, let $(M,d_M)\in\mathscr{E}$.
  Since $(\mathsf{dg}\mathcal{X},\widetilde{\mathcal{Y}})$ is a complete cotorsion pair in $\Di$ by Lemma \ref{prop:complete cotorsion pair}, there is an exact sequence in $\Di$
$$0\to (M,d_{M})\to (Y,d_{Y})\to (X,d_{X})\to 0$$
with $(Y,d_{Y})\in \widetilde{\mathcal{Y}}$ and $(X,d_{X})\in\mathsf{dg}\mathcal{X}.$
Since both $(M,d_M)$ and $(Y,d_Y)$ belong to $\mathscr{E}$, and
$\mathscr{E}$ is thick by \cite[Theorem 4.1]{HJ21}, it follows that $(X,d_{X})$ is exact.
By Lemma \ref{prop:com}, $
\mathsf{dg}\mathcal X\cap\mathscr{E}=\widetilde{\mathcal X}$,
and hence $
(X,d_X)\in\widetilde{\mathcal X}.$
Therefore $(M,d_M)\in\mathcal W$, proving that $
\mathscr{E}\subseteq\mathcal W.$
Thus $\mathcal W=\mathscr{E}.$

Finally, by \cite[Theorem 6.1]{HJ21}, $\mathscr{E}$ is precisely
the class of trivial objects in the model structure defining the
$\Q$-shaped derived category when $\Q$ is the path category of the
Jordan quiver
\vspace{-2.5mm}
\begin{center}
\begin{tikzpicture}[baseline=-0.5ex]
\node (v) at (0,0) {$\bullet$};
\draw[->]
(v) to[out=35,in=-35,looseness=7]
node[right] {$\varepsilon$} (v);
\end{tikzpicture}
\end{center}
\vspace{-2.5mm}
subject to the relation $\varepsilon^{2}=0$.
Consequently, the homotopy category of the above model structure is
equivalent to the corresponding $\Q$-shaped derived category.
\end{proof}



\bibliographystyle{amsplain-nodash}

\def\cprime{$'$}
  \providecommand{\arxiv}[2][AC]{\mbox{\href{http://arxiv.org/abs/#2}{\sf
  arXiv:#2 [math.#1]}}}
  \providecommand{\oldarxiv}[2][AC]{\mbox{\href{http://arxiv.org/abs/math/#2}{\sf
  arXiv:math/#2
  [math.#1]}}}\providecommand{\MR}[1]{\mbox{\href{http://www.ams.org/mathscinet-getitem?mr=#1}{#1}}}
  \renewcommand{\MR}[1]{\mbox{\href{http://www.ams.org/mathscinet-getitem?mr=#1}{#1}}}
\providecommand{\bysame}{\leavevmode\hbox to3em{\hrulefill}\thinspace}
\providecommand{\MR}{\relax\ifhmode\unskip\space\fi MR }
\providecommand{\MRhref}[2]{%
  \href{http://www.ams.org/mathscinet-getitem?mr=#1}{#2}
}
\providecommand{\href}[2]{#2}

\end{document}